\documentclass[preprint]{imsart}
\include{header.tex}

\RequirePackage[OT1]{fontenc}
\RequirePackage{amsthm,amsmath}
\RequirePackage[colorlinks,citecolor=blue,urlcolor=blue]{hyperref}
\usepackage{amssymb}

\startlocaldefs

\mathchardef\varPsi="0109
\newcommand{\rrVert}{\Vert}
\newcommand{\rrvert}{\vert}
\newcommand{\llVert}{\Vert}
\newcommand{\llvert}{\vert}
\newcommand{\trup}[2]{{#1}/{#2}}
\renewcommand{\epsilon}{\varepsilon}
\newcommand{\IR}{\mathbb{R}}
\newcommand{\R}{\mathbb{R}}
\newcommand{\IN}{\mathbb{N}}
\newcommand{\Pb}{\mathbb{P}}
\newcommand{\E}{\mathbb{E}}
\newcommand{\tf}{\mathcal{F}}
\newcommand{\hac}{\mathcal{H}}
\newcommand{\hact}{\mathcal{H}_T}

\newtheorem{remark}{Remark}[section]
\newtheorem{theorem}[remark]{Theorem}
\newtheorem{proposition}[remark]{Proposition}
\newtheorem{lemma}[remark]{Lemma}

\newtheorem{hyp}{Hypothesis}
\endlocaldefs

\begin{document}

\begin{frontmatter}
 
\title{A  proof of a support theorem for stochastic wave equations in H\"older norm with some general noises}

\runtitle{A support theorem for stochastic wave eqs.}

\begin{aug}
\author{\inits{}\fnms{Francisco J.} \snm{Delgado-Vences}\thanksref{e1}\ead[label=e1,mark]{delgado@im.unam.mx}} 
\runauthor{Francisco Delgado-Vences} 
\address{Instituto de Matem\`aticas, Universidad Nacional Aut\'onoma de M\'exico. \\ Oaxaca de Juarez, Oaxaca.\\ \printead{e1}}
\end{aug}

\received{\smonth{7} \syear{2018}}

%
\begin{abstract}
In this paper we characterize the topological support in H\"older norm of the law of the
solution to a stochastic wave equation with three-dimensional space
variable is proved. This note is a continuation of \cite{Delgado--Sanz-Sole012} and \cite{delgado-ss-14}.
The result is a consequence of an approximation theorem, in the convergence of probability, for a sequence of evolution
equations driven by a family of regularizations of the driving noise.
We extend two previous results
on this subject. The first extension is that we cover the case of multiplicative
noise and non-zero initial conditions. The second extension is related to the
covariance function associated to the noise, here we follow the approach of
Hu, Huang and Nualart and ask conditions in terms the of the mean H\"older
continuity of such covariance function.

\end{abstract}


%

\end{frontmatter}


\section{Introduction}
\label{s1}

This paper is an extension and (in some sense) also continuation of \cite{Delgado--Sanz-Sole012}  and \cite{delgado-ss-14}
where we prove a characterization of the topological support in H\"older norm for the law of the solution of a stochastic wave equation. The main difference in the method used here, with respect to the one used in 
\cite{Delgado--Sanz-Sole012} and \cite{delgado-ss-14}, is  very technical and it is 
described in the Remark \ref{rs3.1} below. This fact was noticed in \cite{h-hu-nu} and here we have used very often; it allow us to get very good estimates for several
quantities which leave us to prove H\"older continuity of the solution of the stochastic wave equation as in \cite{h-hu-nu}.

Consider the stochastic partial differential equation (SPDE)
\begin{eqnarray}
\label{s1.1} \biggl(\frac{\partial^2}{\partial t^2} - \Delta \biggr) u(t,x) &=& \varsigma
\bigl(u(t,x) \bigr) \dot M(t,x) + b \bigl(u(t,x) \bigr),
\nonumber
\\[-8pt]
\\[-8pt]
u(0,x) &=& v_0(x)\qquad  \frac{\partial}{\partial t}u(0,x) = \tilde v_0(x),
\nonumber
\end{eqnarray}
%
%
where $\Delta$ denotes the  Laplacian on $\R^3$, $T>0$ is fixed, $t\in\, (0,T]$ and $x\in\R^3$. The 
non-linear terms and the initial conditions are defined by functions $\varsigma, b: \R\rightarrow \R$ 
and $v_0, \tilde v_0: \R^3 \rightarrow \R$, respectively. 

For the definition of the noise we follow \cite{h-hu-nu}.
The notation $\dot M(t,x)$ refers to the 
formal derivative of a Gaussian random field $M$ white in the time variable and with a correlation 
in the space variable given by some function. 
More specifically,
\begin{equation}
\E\left (\dot M(t,x) \dot M(s,y)\right) = \delta_{0}(t-s) f(x-y),
\end{equation}
where $\delta_{0}$ denotes the delta Dirac measure and $f$ is a 
non-negative, non-negative definite function which is a tempered distribution on $\R^3$ and then 
$f$ is locally integrable. We know that in this case $f$ is the Fourier transform of a non-negative tempered measure $\mu \in \R^3$ 
which is called the spectral measure of $f$. That is, for all $\varphi$ belonging to the space $\mathcal{S}(\R^3)$ of rapidly decreasing
$C^\infty$ functions we have 
\begin{equation}
\int_{\R^3} f(x) \varphi(x) dx= \int_{\R^3} \mathcal{F}\varphi(\xi) \mu(d\xi), 
\end{equation}
and there is an integer $m\ge 1$ such that
\begin{equation}
 \int_{\R^3} (1+|\xi|^2)^{-m} \mu(d\xi)<+\infty
\end{equation}
where $\mathcal{F}\varphi$ is the Fourier transform of $\varphi\in\mathcal{S}(\R^3)$:
\begin{equation}
 \mathcal{F}\varphi(\xi)=\int_{\R^3} \varphi(x) e^{-i\xi\cdot x} dx\nonumber
\end{equation}

Our basic assumption on $f$ is 
\begin{equation}\label{cond-f}
 \int_{|x|\le 1} \frac{f(x)}{|x|}dx<+\infty
\end{equation}
It is well-know (see for instance \cite{dalang} or \cite{dalang-quer}) that this is equivalent to 
\begin{equation}
 \int_{|x|\le 1} \frac{\mu(d\xi)}{1+|\xi|^2}<+\infty,
\end{equation}
and since we are in $\R^3$, \eqref{cond-f} is satisfied if there is a $\kappa<2$ such that in a neighborhood of $0$, 
$f(x)\le C |x|^{-\kappa}$.

Let $G(t)$ be the fundamental solution to the wave equation in dimension three, 
$G(t,dx)=\frac{1}{4\pi t}\sigma_t(dx)$,
where $\sigma_t(x)$ denotes the uniform surface measure on the sphere of radius $t>0$ with total mass
$4\pi t^2$.

We consider a random field solution to the SPDE \eqref{s1.1}, which means a real-valued adapted 
(with respect to the natural filtration generated by the Gaussian process $M$) stochastic process
$\{u(t,x), (t,x)\in(0,T]\times \R^3\}$
satisfying 
\begin{eqnarray}
\label{s1.8}
u(t,x)& = & X^0(t,x)+\int_0^t \int_{\mathbb{R}^3} G(t-s,x-y) \varsigma(u(s,y)) M(ds,dy)\nonumber\\ 
&&+ \int_0^t  \big[G(t-s,\cdot)\star b(u(s,\cdot))\big](x) ds.
\end{eqnarray} 
Here
\begin{equation}
\label{i.c}
X^0(t,x)= [G(t)\star \tilde v_0](x)+\left[\frac{d}{d t} G(t)\star v_0\right](x),
\end{equation}

and the symbol ``$\star$" denotes the convolution in the spatial argument. 

The stochastic integral (also termed stochastic convolution) in \eqref{s1.8}  is defined as a 
stochastic integral with respect to a sequence of independent standard Brownian motions 
$\{W_j(s)\}_{j\in\mathbb{N}}$, as follows. Let $\hac$ be the Hilbert space defined by the completion 
of $\mathcal{S}(\R^3)$, endowed with the semi-inner 
product
 \begin{equation}
 \langle\varphi,\psi\rangle_\hac = \int_{\R^3}\mu(d\xi)\tf\varphi(\xi) \overline{\tf\psi(\xi)}=
 \int_{\R^3}\int_{\R^3} \varphi(y)\psi(x)f(x-y)dxdy,\nonumber
 \end{equation} 
where $\mu$ is the spectral measure of $f$. Then
\begin{eqnarray}
\label{s1.7}
&&\int_0^t \int_{\mathbb{R}^3} G(t-s,x-y) \varsigma(u(s,y)) M(ds,dy)\nonumber\\
&&\qquad :=\sum_{j\in \mathbb{N}}\int_0^t \langle G(t-s,x-\ast) \varsigma(u(s,\ast)), e_j\rangle_{\hac} W_j(ds),
\end{eqnarray}
where  $(e_j)_{j\in\mathbb{N}}\subset \mathcal{S}(\R^3)$ is a complete orthonormal basis of $\hac$. 

Assume that $\varphi\in \hac$ is a signed measure with finite total variation. Then, by applying \cite[Theorem 5.2]{kx}
(see also \cite[Lemma 12.12]{mat}  for the case of probability measures with compact support) and a 
polarization argument on the positive and negative parts of $\varphi$, we obtain
\begin{equation}
\label{fundamental}
\Vert\varphi\Vert_\hac^2=C\int_{\R^3}\int_{\R^3} \varphi(dx) \varphi(dy) f(x-y)=C\int_{\R^3}\mu(d\xi)\tf\varphi(\xi).
\end{equation}

For $t_0\in[0,T]$,  $K\subset\R^3$ compact and $\rho\in(0,1)$, we denote by  $\mathcal{C}^\rho([t_0,T]\times K)$ 
the space of real functions $g$ such that $\Vert g\Vert_{\rho,t_0,K}<\infty$, where
\begin{equation}
\Vert g\Vert_{\rho,t_0,K}:= \sup_{(t,x)\in[t_0,T]\times K} |g(t,x) |+  
\sup_{\stackrel{(t,x),(\bar{t},\bar{x})\in[t_0,T]\times K}{(t,x)\ne(\bar t,\bar x)}} 
\frac{|g(t,x)-g(\bar{t},\bar{x})| }{(|t-\bar{t}|+|x-\bar{x}|)^\rho}.\nonumber
\end{equation} 

Let $0<\rho^\prime<\rho$ and $\mathcal{E}^{\rho^\prime}([t_0,T]\times K)$ be the space of H\"older 
continuous functions $g$ of degree $\rho^\prime$ such that
\begin{equation}
O_g(\delta):=\sup_{|t-s|+|x-y|<\delta}\frac{\vert
g(t,x)-g(s,y)\vert}{(|t-s|+|x-y|)^{\rho^\prime}} \rightarrow 0, \mbox{ if } \delta\to 0.
\end{equation}
The space $\mathcal{E}^{\rho^\prime}([t_0,T]\times K)$ endowed with the norm $\Vert\cdot\Vert_{\rho^\prime,t_0,K}$ is a Polish 
space and the embedding  $\mathcal{C}^\rho([t_0,T]\times K)\subset \mathcal{E}^{\rho^\prime}([t_0,T]\times K)$ is compact.

Assume that the functions $\varsigma$ and $b$ are Lipschitz continuous and the initial conditions $v_0$, 
$\tilde v_0$ satisfy the assumption a) fro Hypothesis {\bf \ref{HypH1}} below. When $f$ is the Riesz kernel,
i.e. $f(x)=|x|^{-\beta}$, with $\beta \in ]0,2[$, Theorem 4.11 in \cite{dss}
along with \cite[Proposition 2.6]{dalang-quer} give the existence of a random field solution to \eqref{s1.8}
with sample paths in $\mathcal{C}^\rho([0,T]\times K)$, with $\rho\in\left(0,\gamma_1\wedge \gamma_2\wedge\frac{2-\beta}{2}\right)$.
In our case one can prove the existence and uniqueness of the solution to \eqref{s1.1} in the same way as in 
Theorem 4.3 in \cite{dalang-quer}; indeed, in the appendix we sketch a proof of such result (see Teorema \ref{ts5.1}).

For any $t\in(0,T]$, let $\hact=L^2([0,t]; \hac)$. Fix $h\in\hact$ and consider the deterministic evolution equation
\begin{align}
 \label{sm.h}
\Phi^h(t,x) &= X^0(t,x) + \left\langle G(t-\cdot,x-\ast)\varsigma(\Phi^h(\cdot,\ast)) ,h\right\rangle_{\mathcal{H}_t}\nonumber\\
&+\int_0^t ds [G(t-s,\cdot)\star(\Phi^h(s,\cdot))](x).
\end{align}
The main objective in \cite{Delgado--Sanz-Sole012} is to prove that, in the particular case $v_0=\tilde v_0=0$, 
the topological support of the law of the solution to \eqref{s1.8} in the space  $\mathcal{E}^\rho([t_0,T]\times K)$ with
$\rho\in\left(0,\frac{2-\beta}{2}\right)$ is the closure in the H\"older norm of the set $\{\Phi^h, h\in \hac_T\}$, for any
$t_0>0$. (see \cite[Theorem 3.1]{Delgado--Sanz-Sole012}). In \cite{delgado-ss-14} they obtained a similar result for the case of non zero
initial conditions but restricted to the case of the function $\varsigma$ being a linear function, this is called the affine case.
Their approach is based on the fractional Sobolev imbedding theorem and the Fourier transform technique.

The aim of this paper is to prove a extension of this result allowing non null initial conditions $v_0$, $\tilde v_0$, 
with a multiplicative noise (i.e. we will let $\varsigma$ to be globally Lipschitz) and the covariance functional will 
satisfy the  {\bf Hypothesis \ref{HypH1}.(b)} and {\bf Hypothesis \ref{HypH2}.(a),(c)} (defined below).  Notice that with these assumptions we will
cover the case of the Riesz Kernel (see \cite{h-hu-nu} for more details).The theorem is stated below.

\medskip

We will use the following two sets of hypotheses (see \cite{h-hu-nu}).

\begin{hyp}\label{HypH1}
\begin{itemize}
 \item[{\bf (a)}] $v_o\in C^2(\R^3)$, $v_0$, $\triangledown v_0$ are bounded and $\triangle v_0$ and $\bar{v}_0$ are H\"older 
continuous of orders $\gamma_1$ and $\gamma_2$ respectively, $\gamma_1,\gamma_2\in ]0,1]$.
\\[-8pt]
\item[{\bf (b)}] The function $f$ satisfies condition \eqref{cond-f} and for some $\gamma\in ]0,1]$ and $\gamma'\in ]0,2]$ we have
for all $w\in\R^3$ 
\begin{equation}\label{H1-c}
 \int_{|z|\le 2T} \frac{|f(z+w)-f(z)|}{|z|}dz \le C |w|^\gamma
\end{equation}
and 
\begin{equation}\label{H1-d}
 \int_{|z|\le 2T} \frac{|f(z+w)-2f(z)+f(z-w)|}{|z|}dz \le C |w|^{\gamma'}
\end{equation}
\end{itemize}
 
\end{hyp}

For the H\"older continuity in time we will use, in adition, the followig set of hypothesis.
Let $S^2$ denote the unit sphere in $\R^3$ and $\sigma(d\xi)$ the uniform measure on it.

\begin{hyp}\label{HypH2}
\begin{itemize}
 \item[{\bf (a)}] For some $0<\nu\le 1$, 
 \begin{equation}\label{H2-a}
  \int_{|z|\le h} \frac{f(z)}{|z|}dz \le C |h|^\nu\qquad \mbox{ for any } 0<h\le 2T
 \end{equation}
\\[-8pt]
\item[{\bf (b)}]  For some $0<\kappa_1\le 1$ and for any $q\ge 2$ and $t\in(0,T]$ we have 
\begin{equation}\label{H2-b}
 \E\big(|u(t,x)-u(t,y)|^q 1_{L_n(t)}\big)\le C |x-y|^{q\kappa_1}.
\end{equation}
where $L_n(t)$ is defined in \eqref{localization}. 
\\[-8pt]
\item[{\bf (c)}] Let $\xi$ and $\eta$ unit vectors in $\R^3$ and $0<h\le 1$. We have
\begin{equation}\label{H2-c1}
 \int_{0}^T\int_{S^2} \int_{S^2} \Big|f\big(s(\xi+\eta)+ h(\xi+\eta)\big)-f\big(s(\xi+\eta)+
 h\eta\big)\Big| s\sigma(d\xi) \sigma(d\eta)ds\le C |h|^{\rho_1},
\end{equation}
for some $\rho_1\in ]0,1]$, and 
\begin{eqnarray}\label{H2-c2}
 &&\int_{0}^T\int_{S^2} \int_{S^2} \Big|f\Big(s(\xi+\eta)+ h(\xi+\eta)\Big)-f\Big(s(\xi+\eta)+h\xi\Big)\nonumber\\
 &&\qquad\qquad -f\Big(s(\xi+\eta)+h\eta\Big)+f\Big(s(\xi+\eta)\Big)\Big| s^2\sigma(d\xi) \sigma(d\eta)ds\le C |h|^{\rho_2},
\end{eqnarray}
for some $\rho_2\in ]0,2]$.
\end{itemize}
 
\end{hyp}

Notice that when we prove the results, first we will prove the H\"older continuity in space and after that we will prove the 
 H\"older continuity in time, so at the point we use {\bf (b)} of  the  {\bf hypothesis 2} we will have established it.


\begin{theorem}
\label {sth}
Assume that
\begin{description}
\item {(1)} the functions $\varsigma$ and $b$ are Lipschitz continuous;
\item {(2)} Assume  {\bf Hypothesis 1} and {\bf Hypothesis 2} hold.
\end{description}
Fix  $t_0>0$ and a compact set $K\subset \R^3$. Then the topological support of the law of the solution to \eqref{s1.8} 
in the space  $\mathcal{E}^\rho([t_0,T]\times K)$ with $\rho \in\, \Big]0,\min\Big(\gamma_1,\gamma_2,\gamma,
\tfrac{\gamma'}{2},\tfrac{\nu+1}{2},
\tfrac{\rho_1+\kappa}{2},\tfrac{\rho_2}{2}\Big) \Big[$ 
is the closure in the H\"older norm $\Vert\cdot\Vert_{\rho,t_0,K}$ of the set $\{\Phi^h, h\in \hac_T\}$, where 
$\Phi^h$ is given in \eqref{sm.h} and $\kappa\in\, \Big]0,\min\Big(\gamma_1,\gamma_2,\gamma,
\tfrac{\gamma'}{2}\Big) \Big[$.
\end{theorem}
 After the seminal paper \cite{stroock},
an extensive literature on support theorems for stochastic differential equations appeared (see for example, 
\cite{aida-kusuoka-stroock}, \cite{g-n-ss}, \cite{millet-ss94b}, and references 
herein). The analysis of the uniqueness of invariant measures is one of the motivations for the characterization 
of the support of stochastic evolution equations (see \cite[Section 1]{Delgado--Sanz-Sole012} for more details). 
 
 As in \cite{Delgado--Sanz-Sole012,delgado-ss-14}, Theorem \ref{sth} will be a corollary of a general result on approximations of 
 Equation \eqref{s1.8} by a sequence of SPDEs obtained by smoothing the noise $M$. The precise statement, given in 
 Theorem \ref{ts3.1}, provides a Wong-Zakai type theorem in H\"older norm. The method 
 relies on \cite{aida-kusuoka-stroock}, further developed and used in
 \cite{Ba-Mi-SS}, \cite{g-n-ss}, \cite{millet-ss94a}, 
 \cite{millet-ss94b}, \cite{milletss2}. We refer the reader to  \cite[Section 1]{Delgado--Sanz-Sole012}
 for a detailed description of the method for the proof of support theorems based on approximations.

As in \cite{Delgado--Sanz-Sole012} (see also \cite{delgado-ss-14}) we apply the approximation method of \cite{millet-ss94b}
to obtain a characterization of the topological support
of the law of
$u$ (the solution to \eqref{s1.8}) in the H\"older norm $\Vert\cdot
\Vert_{\rho,t_0,K}$. The core of the work consists of an
approximation result for a family of equations more general than
equation \eqref{s1.8} by a sequence of pathwise evolution equations
obtained by a smooth approximation of the driving process $M$. In
finite dimensions, the celebrated
Wong--Zakai approximations for diffusions in the supremum norm could be
considered as an analogue. However there are two substantial
differences, first the type of equation we consider in this paper is
much more complex, and moreover we deal with a stronger topology.

For the sake of completeness, we give a brief description of the
procedure of \cite{millet-ss94a} in the particular context of this
work, and refer the reader to \cite{millet-ss94a} for further details.

Let $(\bar\Omega,\bar{\mathcal{G}}, \bar\mu)$ be the canonical
space of a standard real-valued Brownian motion on $[0,T]$. In the
sequel, the reference probability space will be
$(\Omega,\mathcal{G}, \Pb):=(\bar\Omega^{\mathbb{N}},\bar
{\mathcal{G}}^{\otimes\mathbb{N}}, \bar\mu^{\otimes\mathbb
{N}})$. By the preceding identification of $M$ with $(W_j, j\in\mathbb
{N})$, this is the canonical probability space of $M$.

Assume that there exists a measurable mapping $\xi_1: L^2 ([0,T];
{\ell}^2 )\rightarrow\mathcal{C}^\rho([t_0,T]\times K)$, and
a sequence $w^n:  \Omega\rightarrow L^2 ([0,T]; {\ell}^2 )$
such that for every $\epsilon>0$,\vspace*{1.5pt}
%
\begin{equation}
\label{s1.18} \lim_{n\to\infty}\Pb \bigl\{\bigl \Vert u-\xi_1
\bigl(w^n\bigr)\bigr \Vert_{\rho
,t_0,K}>\epsilon \bigr\}=0.
\end{equation}
Then $\operatorname{supp}(u\circ\Pb^{-1}) \subset\overline{\xi_1
(L^2 ([0,T]; {\ell}^2 ) )}$, where the closure refers
to the H\"older norm $\Vert\cdot\Vert_{\rho,t_0,K}$.

Next, we assume that there exists a mapping $\xi_2:  L^2 ([0,T];
{\ell}^2 )\rightarrow\mathcal{C}^\rho([t_0,T]\times K)$ and
for any $h\in L^2 ([0,T]; {\ell}^2 )$, we suppose\vadjust{\goodbreak} that
there exist a sequence $T_n^h:  \Omega\rightarrow\Omega$ of
measurable transformations such that, for any $n\ge1$, the probability
$\Pb\circ(T_n^h)^{-1}$ is absolutely continuous with respect to $\Pb
$ and, for any $h\in L^2 ([0,T]; {\ell}^2 )$, $\epsilon>0$,
%
\begin{equation}
\label{s1.19} \lim_{n\to\infty}\Pb \bigl\{\bigl \Vert u
\bigl(T_n^h\bigr)-\xi_2(h)
\bigr \Vert_{\rho,t_0,K}<\epsilon \bigr\}>0.
\end{equation}
Then $\operatorname{supp}(u\circ\Pb^{-1})\supset\overline{\xi_2
(L^2 ([0,T]; {\ell}^2 ) )}$.

For any $h\in L^2 ([0,T]; {\ell}^2 )$ (or equivalently,
$h\in\hact$), recall the deterministic evolution equation $\Phi^h(t,x)$ given by \eqref{sm.h} 
and similarly as for $u$, the mapping $(t,x)\in[t_0,T]\times K\mapsto\Phi
^h(t,x)$ belongs to $\mathcal{C}^\rho([t_0,T]\times K)$.

Let $\xi_1(h)=\xi_2(h)=\Phi^h$, and $(w^n)_{n\ge1}$ be given by
\eqref{s3.3}.
From \eqref{s3.4} and the isometric representation of $\hact$, we see that
$w^n:  \Omega\rightarrow L^2 ([0,T]; {\ell}^2 )$.
Given $h\in L^2 ([0,T]; {\ell}^2 )$, we define
%
\begin{equation}
\label{sm.1} T_n^h(\omega) = \omega+h-w^n.
\end{equation}
By Girsanov's theorem, the probability $\Pb\circ(T_n^h)^{-1}$ is
absolutely continuous with respect to $\Pb$.

According to \eqref{s1.18}, \eqref{s1.19}, the final objective is to
prove that
\[
\lim_{n\to\infty}\Phi^{w^n}=u,\qquad \lim
_{n\to\infty}u\bigl(T_n^h\bigr)=
\Phi^h,
\]
in probability and with the H\"older norm $\Vert\cdot\Vert_{\rho
,t_0,T}$. Then, by the preceding discussion we infer that
the support of the law of $u$ in the H\"older norm is the closure of
the set of functions $\{\Phi^h, h\in\hact\}$ (see Theorem~\ref
{tsm.1} for the rigorous statement).
Notice that the characterization of the support does not depend on the
approximating sequence $(w^n)_{n\in\mathbb{N}}$.

The structure of this paper is exactly the same as the one in  \cite{Delgado--Sanz-Sole012} and it is structured as follows. 
The next Section~\ref{s3} is
devoted to a general approximation result (see \cite{Delgado--Sanz-Sole012} for the details of the approximation).
Section~\ref{sm} is devoted to the proof of the characterization of
the support of $u$. It is a corollary of Theorem~\ref{ts3.1}.
Section~\ref{s4} is of technical character. It is devoted to establish some
auxiliary results which are needed in some proofs of Section~\ref{s3}.
In the \hyperref[s5]{Appendix},
a theorem on existence and uniqueness of a random
field solution for a quite general evolution equation is proved. It
provides the rigorous setting for all the stochastic partial
differential equations that appear in this paper. The section also
contains a known but fundamental result used at some crucial parts
of the proofs of Sections~\ref{s3} and~\ref{sm}.

As was formulated in \cite{Ba-Mi-SS}, and further developed in \cite
{milletss2}, there are two main elements in the proof of Theorem~\ref
{ts3.1}: a control on the $L^p(\Omega)$-increments in time and in
space of the processes $X$ and $X_n$, independently of $n$, and
$L^p(\Omega)$ convergence of $X_n(t,x)$ to $X(t,x)$, for any $(t,x)$.
The precise assertions are given in Theorems~\ref{ts3.2} and~\ref
{ts3.3}, respectively.

For an It\^o's stochastic differential equation, smoothing the noise
leads to a Stratonovich (or pathwise) type integral, and the correction
term between the two kinds of integrals appears naturally in the
approximating scheme. In our setting, correction terms explode and
therefore they must be avoided. Instead, a control on the growth of the
regularized noise is used.
This method was introduced in \cite{milletss2} and successfully
used in \cite{Delgado--Sanz-Sole012}  and  \cite{delgado-ss-14}.  Here we use it too. 
The control is achieved by introducing a localization
in $\Omega$ (see \eqref{localization}). With this method, the
convergence of the approximating sequence $X_n$ to $X$ takes place in
probability.

Throughout the paper, we shall often call different positive and finite
constants by the same notation, even if they differ from one place to another.

\section{Approximations of the wave equation}
\label{s3}

Consider smooth approximations of $W$ defined as follows. Fix $n\in\IN
$ and consider the partition
of $[0,T]$ determined by $\frac{iT}{2^n}$, $i=0,1,\ldots,2^n$. Denote
by $\Delta_i$ the interval $[\frac{iT}{2^n},
\frac{(i+1)T}{2^n}[$ and by $|\Delta_i|$ its length. We write
$W_j(\Delta_i)$ for the increment $W_j(\frac{(i+1)T}{2^n})-
W_j(\frac{iT}{2^n})$, $i=0,\ldots,2^n-1$, $j\in\IN$. Define
differentiable approximations of $(W_j, j\in\mathbb{N})$ as follows:
\[
W^n= \biggl(W_j^n=\int
_0^\cdot\dot{W}_j^n(s)
\,\mathrm{d}s, j\in\IN \biggr),
\]
where for $j>n$, $\dot{W}_j^n=0$, and for $1\le j\le n$,

\[
\dot{W}_j^n(t)= %
\begin{cases}
 \sum
_{i=0}^{2^n-2} 2^{n}T^{-1}W_j(
\Delta_i)1_{\Delta
_{i+1}}(t) &\quad \text{ if } t\in
\bigl[2^{-n}T,T\bigr],
\\
0 &\quad \text{ if } t \in\bigl [0,2^{-n}T\bigr [.  %
\end{cases}
\]
Set
%
\begin{equation}
\label{s3.3} w^n(t,x)=\sum_{j\in\IN}
\dot{W}_j^n(t) e_j(x).
\end{equation}
It is easy to check that, for any $p\in[2,\infty[$,
%
\begin{equation}
\label{s3.4} \bigl \|w^n\bigr \|_{L^p(\Omega,\mathcal{H}_T)}\le C n^{{\trup{1}{2}}}2^{n/2}.
\end{equation}
In particular, from \eqref{s3.4} it follows that $w^n$ belongs to
$\mathcal{H}_T$ a.s.

In this section, we shall consider the equations
%
\begin{eqnarray}
\label{s3.7}X(t,x) &=&X^0(t,x)+
\int_0^t \int
_{\R^3} G(t-s,x-y) (A+B) \bigl(X(s,y)\bigr) M(\mathrm{d}s,
\mathrm{d}y)
\nonumber
\\
&&{}+\bigl\langle G(t-\cdot,x-\ast)D\bigl(X(\cdot,\ast)\bigr),h\bigr
\rangle_{\mathcal{H}_t}
\\
&&{}+\int_0^t \bigl[G(t-s,
\cdot)\star b\bigl(X(s,\cdot)\bigr) \bigr](x)\, \mathrm{d}s,
\nonumber
\\
\label{s3.6}X_n(t,x) &=&X^0(t,x)+\int_0^t
\int_{\R^3} G(t-s,x-y) A\bigl(X_n(s,y)\bigr) M(
\mathrm{d}s,\mathrm{d}y)
\nonumber
\\
&&{} +\bigl\langle G(t-\cdot,x-\ast)B\bigl(X_n(\cdot,\ast)
\bigr),w^n\bigr\rangle_{\mathcal{H}_t}
\nonumber
\\[-8pt]
\\[-8pt]
&&{}+\bigl\langle G(t-\cdot,x-
\ast)D\bigl(X_n(\cdot,\ast)\bigr),h \bigr\rangle_{\mathcal{H}_t}\nonumber
\\
&&{} +\int_0^t \bigl[G(t-s,\cdot)\star b
\bigl(X_n(s,\cdot)\bigr) \bigr](x) \,\mathrm{d}s,
\nonumber
\end{eqnarray}
where $n\in\IN$, $h\in\mathcal{H}_T$, $w^n$ defined as in (\ref{s3.3}) and $A, B, D, b: \R\to\R$.

Moreover, we also need the slight modification of these equations
defined by
%
\begin{eqnarray}
\label{s3.8.2} X_n^-(t,x) &=&X^0(t,x)+\int_0^{t_n}
\int_{\R^3} G(t-s,x-y) A\bigl(X_n(s,y)\bigr) M(
\mathrm{d}s,\mathrm{d}y)
\nonumber
\\
&&{} +\bigl\langle G(t-\cdot,x-\ast)B\bigl(X_n(\cdot,\ast)
\bigr)1_{[0,t_n]}(\cdot ),w^n\bigr\rangle_{\mathcal{H}_t}
\nonumber
\\[-8pt]
\\[-8pt]
&&{} +\bigl\langle G(t-\cdot,x-\ast)D\bigl(X_n(\cdot,\ast)
\bigr)1_{[0,t_n]}(\cdot ),h\bigr\rangle_{\mathcal{H}_t}
\nonumber
\\
&&{} +\int_0^{t_n} \bigl[G(t-s,\cdot)\star b
\bigl(X_n(s,\cdot)\bigr) \bigr](x)\, \mathrm{d}s,
\nonumber
\\
\label{s3.8.3} X(t,t_n,x) &=&X^0(t,x)+\int_0^{t_n}
\int_{\R^3} G(t-s,x-y) (A+B) \bigl(X(s,y)\bigr) M(\mathrm{d}s,
\mathrm{d}y)
\nonumber
\\
&&{} +\bigl\langle G(t-\cdot,x-\ast)D\bigl(X(\cdot,\ast)\bigr)1_{[0,t_n]}(
\cdot ),h\bigr\rangle_{\mathcal{H}_t}
\\
&&{} +\int_0^{t_n} \bigl[G(t-s,\cdot)\star b
\bigl(X(s,\cdot)\bigr) \bigr](x) \,\mathrm{d}s,
\nonumber
\end{eqnarray}
where for any $n\in\IN$, $t\in[0,T]$, $t_n= \max\{\underline{t}_n-
2^{-n}T, 0 \}$, with
%
\begin{equation}
\label{s3.8.1} \underline{t}_n = \max \bigl\{k2^{-n}T,
k=1,\ldots,2^n-1: k2^{-n}T\le t \bigr\}.
\end{equation}

We will consider the following assumption during all paper.
\begin{hyp}\label{HypB}
The coefficients $A,B,D,b: \R\mapsto\R$ are
globally Lipschitz continuous.
 
\end{hyp}

Notice that equation \eqref{s3.6} is more general than \eqref{s3.7}
and \eqref{s1.8}. In Theorem~\ref{ts5.1}, we prove a result on
existence and uniqueness of a random field solution to
a class of SPDEs which applies to equation \eqref{s3.6}.

\begin{remark}
\label{rs3.1}
In opposition to Remark A.2 in \cite{Delgado--Sanz-Sole012}, here we have not the translation invariance of the moments.
Following \cite{h-hu-nu}, we have the following facts that will be used at several points in the proofs of the results of 
the following section. 

For any $z\in\R^3$, set $z_1=x-z$ and $z_2=y-z$ then trivially $z_1-z_2=x-y$. Note the folowing.
First, we see that 
\begin{eqnarray}
\label{s3.est} \sup_{z\in \R^3}\E \bigl(\bigl \vert X(t,x-z)-X(t,y-z)\bigr \vert^p \bigr)&=&
\sup_{z_1-z_2=x-y}\E \bigl(\bigl \vert X(t,z_1)-X(t,z_2)\bigr \vert^p \bigr),
\end{eqnarray}

Secondly, from $z_1=x-z$ and $z_2=y-z$ we have $x=z_1+z$ and $y=z_2+z$, then
\[
 \E \bigl(\bigl \vert X(t,x)-X(t,y)\bigr \vert^p \bigr)=\E \bigl(\bigl \vert X(t,z_1+z)-X(t,z_2+z)\bigr \vert^p \bigr)
\]
and again $x-y=z_1-z_2$ then 

\begin{eqnarray}
\label{s3.est1} \sup_{z\in \R^3}\E \bigl(\bigl \vert X(t,z_1+z)-X(t,z_2+z)\bigr \vert^p \bigr)&=&
\sup_{x-y=z_1-z_2 }\E \bigl(\bigl \vert X(t,x)-X(t,y)\bigr \vert^p \bigr),
\nonumber
\\[-8pt]
\\[-8pt]
&=&\sup_{z_1-z_2=x-y} \E \bigl(\bigl \vert X(t,z_1)-X(t,z_2)\bigr \vert^p \bigr),
\nonumber
\end{eqnarray}

for any $x,y,z_1,z_2\in\IR^3$ and any $p\in[1,\infty[$. Consequently, a
similar property also holds for $X_n^-(t,\ast)$ and $X_n(t,t_n,\ast)$
defined in \eqref{s3.8.2}, \eqref{s3.8.3}, respectively
\end{remark}

The aim of this section is to prove the following theorem.

\begin{theorem}
\label{ts3.1}
We assume {\bf (Hypothesis 1)} and {\bf Hypothesis 2} holds. Fix $t_0>0$ and a compact set $K\subset\R
^3$. Set $\gamma,\gamma_1,\gamma_2,\rho_1,\nu\in ]0,1[ $,\,\, $\gamma',\rho_2\in ]0,2[$ and
$\kappa \in\, \Big]0,\min\Big(\gamma_1,\gamma_2,\gamma,\tfrac{\gamma'}{2}\Big) \Big[$. Then for any
$\rho \in\, \Big]0,\min\Big(\gamma_1,\gamma_2,\gamma,\tfrac{\gamma'}{2},\tfrac{\nu+1}{2},
\tfrac{\rho_1+\kappa}{2},\tfrac{\rho_2}{2}\Big) \Big[$
and $\lambda>0$,
%
\begin{equation}
\label{s3.9} \lim_{n\to\infty} \Pb \bigl(\|X_n-X
\|_{\rho,t_0,K}> \lambda \bigr)=0.
\end{equation}
\end{theorem}

The convergence \eqref{s3.9} will be proved through several steps.
The main ingredients are local $L^p$ estimates of increments of $X_n$
and $X$, in time and in space, and a local $L^p$
convergence of the sequence $X_n(t,x)$ to $X(t,x)$.

Let us describe the \textit{localization} procedure (see \cite
{milletss2}). Fix $\alpha>0$. For any integer $n\ge1$ and $t\in
[0,T]$, define
%
\begin{equation}
\label{localization} L_{n}(t)= \Bigl\{\sup_{1\le j\le n} \sup
_{0\le i\le[2^nt
T^{-1}-1]^+} \bigl |W_j(\Delta_i)\bigr |\le\alpha
n^{\trup{1}{2}}2^{-\trup
{n}{2}} \Bigr\},
\end{equation}
where $\alpha>(2\ln2)^{\trup{1}{2}}$. Notice that the sets $L_n(t)$
decrease with $t\ge0$. Moreover, in \cite{milletss2}, Lemma~2.1, it is
proved that $\lim_{n\to\infty} \Pb(L_n(t)^c)=0$.

It is easy to check that
%
\begin{equation}
\label{s3.101} \bigl \Vert w^n(t,\ast)1_{L_n(t)}
\bigr \Vert_{\mathcal{H}}\le C n^{\trup{3}{2}} 2^{\trup{n}{2}}.
\end{equation}
Moreover, for any $0\le t\le t'\le T$
\[
\bigl\llVert w^{n}1_{L_n(t')}1_{[t,t']}\bigr\rrVert
_{\mathcal{H}_T}\le Cn^{\trup{3}{2}}2^{\trup{n}{2}}\bigl |t'-t\bigr |^{\trup{1}{2}}.
\]
In particular, if $[t,t']\subset\Delta_i$ for some $i=0,\ldots
,2^n-1$, then
%
\begin{equation}
\label{s3.14} \bigl\llVert w^{n}1_{L_n(t')}1_{[t,t']}
\bigr\rrVert _{\mathcal{H}_T}\le Cn^{\trup{3}{2}}.
\end{equation}


As has been announced in the \hyperref[s1]{Introduction}, the proof of Theorem~\ref
{ts3.1} will follow from Theorems~\ref{ts3.2} and~\ref{ts3.3} below.
We denote by $\Vert\cdot\Vert_p$ the
$L^p(\Omega)$ norm.
%
\begin{theorem}
\label{ts3.2}
We assume {\bf (Hypothesis 1)} and {\bf Hypothesis 2} holds. Fix $t_0\in\,]0, T[$ and a compact subset
$K\subset\IR^3$. Let $t_0\le t\le\bar{t}\le T$, $x,\bar{x}\in K$. Set Set $\gamma,\gamma_1,\gamma_2,\rho_1,\nu\in ]0,1[ $,
\,\, $\gamma',\rho_2\in ]0,2[$ and  $\kappa \in\, 
\Big]0,\min\Big(\gamma_1,\gamma_2,\gamma,\tfrac{\gamma'}{2}\Big) \Big[$.
Then, for any
$p\in[1,\infty)$ and any $\rho \in\, \Big]0,\min\Big(\gamma_1,\gamma_2,\gamma,\tfrac{\gamma'}{2},\tfrac{\nu+1}{2},
\tfrac{\rho_1+\kappa}{2},\tfrac{\rho_2}{2}\Big) \Big[$,
there exists a positive constant $C$ such that
%
\begin{equation}
\label{s3.15} \sup_{n\ge1} \bigl\llVert \bigl[X_n(t,x)-X_n(
\bar{t},\bar{x}) \bigr] 1_{L_n(\bar{t})} \bigr\rrVert _p \le C \bigl( |
\bar{t}- t |+ |\bar{x}-x | \bigr)^{\rho}.
\end{equation}
\end{theorem}

\begin{theorem}
\label{ts3.3}
The assumptions are the same as in Theorem~\ref{ts3.2}.
Fix $t\in[t_0,T]$, $x\in\IR^3$. Then, for any
$p\in[1,\infty)$
%
\begin{equation}
\label{s3.16} \lim_{n\to\infty} \mathop{\sup_{{t\in[t_0,T]}}}_{{x\in K(t)}}\bigl \|
\bigl(X_n(t,x)-X(t,x) \bigr) 1_{L_n(t)}\bigr \|_p =0,
\end{equation}
where for $t\in[0,T]$,
\[
K(t)= \bigl\{x\in\R^3: d(x,K)\le T-t\bigr\},
\]
and $d$ denotes the Euclidean distance.
\end{theorem}

The proof of Theorem~\ref{ts3.2} is carried out through two steps.
First, we shall consider $t=\bar t$ and obtain \eqref{s3.15},
uniformly in $t\in[t_0,T]$. Using this, we will consider $x=\bar x$
and establish \eqref{s3.15}, uniformly in $x\in K$. We devote the next
two subsections to the proof of these results.

\subsection{Increments in space}
\label{ss3.1}

Throughout this section, we fix $t_0\in\,]0,T[$ and a compact set
$K\subset\IR^3$. The objective is to prove the following proposition.
%
\begin{proposition}
\label{pss3.1.1}
Suppose that {\bf Hypothesis 1} holds. Fix $t\in[t_0,T]$ and
$x,\bar{x}\in K$. Then, for any $p\in[1,\infty)$ and $\rho \in\, \Big]0,
\min\Big(\gamma_1,\gamma_2,\gamma,\tfrac{\gamma'}{2}\Big) \Big[$, 
there exists a finite constant $C$ such that
%
\begin{equation}
\label{s3.17} \sup_{n\ge0} \sup_{t\in[t_0,T]}\bigl \|
\bigl(X_n(t,x)-X_n(t,\bar{x}) \bigr) 1_{L_n(t)}
\bigr \|_p \le C |x-\bar{x}|^{\rho}.
\end{equation}
\end{proposition}

In the next lemma, we give an abstract result that will be used
throughout the proofs. We start by introducing some notation.

For a function $f: \R^3 \to\R$, we set
\begin{eqnarray*}
Df(u,x) &=&f(u+x)-f(u),
\\
\bar{D}^2f(u,x,y) &=&f(u+x+y)-f(u+x)-f(u+y)+f(u),
\\
D^2f(u,x) &=&\bar{D}^2f(u-x,x,x) =f(u-x)-2f(u)+f(u+x).
\end{eqnarray*}

\begin{lemma}
\label{lss3.1.1}
Consider a sequence of predictable stochastic processes $\{Z_n(t,x),
(t,x)\in[0,T]\times\R^3\}$, $n\in\mathbb{N}$, such that, for any
$p\in[2,\infty[$,
%
\begin{equation}
\label{s3.21} \sup_n \sup_{(t,x)\in[0,T]\times\R^3} \E
\bigl(\bigl\llvert Z_n(t,x)\bigr\rrvert ^p \bigr)< C,
\end{equation}
for some finite constant $C$.
For any $t\in[0,T]$, $x,\bar x\in\R^3$, we define
\[
I_n(t,x,\bar{x}) := \int_0^t
\mathrm{d}s \bigl \Vert Z_n(s,\ast) \bigl[G(t-s,x-\ast )-G(t-s,\bar{x}-\ast)
\bigr]\bigr  \Vert^2_{\mathcal{H}}.
\]
Then, for any $p\in[2,\infty[$,
%
\begin{eqnarray}
\label{s3.220} \E \bigl( \bigl |I_n(t,x,\bar{x}) \bigr |^{{\trup{p}{2}}} \bigr)
 &\le& C \biggl\{|x-\bar{x}|^{\gamma p}+|x-\bar{x}|^{\gamma' p/2}
\nonumber
\\
&&\qquad+\int_0^t\mathrm{d}s \Bigl[\sup_{z_1-z_2=x-\bar{x}} \E \bigl( \bigl |Z_n(s,z_1)-Z_n(s,
z_2)\bigr  |^p \bigr) \Bigr]\biggl\}
\end{eqnarray}
where $\gamma \in\,]0,1] $ and $\gamma' \in\,]0,2]$.
\end{lemma}
\begin{proof}
Set $w=x-\bar x$.
First, we notice that $I_n(t,x,\bar x)$ is the second order moment of
the stochastic integral
\[
\int_0^t \int_{\R^3}
Z_n(s,y)\bigl[G(t-s,x-y)-G(t-s,\bar x-y)\bigr] M(\mathrm{d}s,
\mathrm{d}y).
\]

We write $I_n(t,x,\bar x)$ using \eqref{fundamental}. This yields
\begin{eqnarray*}
I_n(t,x,\bar x)&=& C\int_0^t
\mathrm{d}s \int_{\R^3} \int_{\R^3}
Z_n(s,u)Z_n(s,v) \bigl[G(t-s,x-\mathrm{d}u)- G(t-s,\bar
x-\mathrm{d}u)\bigr]
\\
&&{} \times\bigl[G(t-s,x-\mathrm{d}v)-G(t-s,\bar x-\mathrm{d}v)
\bigr]f(u-v).
\end{eqnarray*}
Then, as in \cite{dss}, pages 19--20, we see that, by decomposing this
expression into the sum of four integrals, by applying a change of
variables and rearranging terms, we have
\[
I_n(t,x,\bar x)= C\sum_{i=1}^4
J_i^t(x,\bar{x}),
\]
where, for $i=1,\ldots,4$,
\[
J_i^t(x,\bar{x})=\int_0^t
\mathrm{d}s \int_{\R^3}\int_{\R^3} G(s,
\mathrm{d}u)G(s,\mathrm{d}v) h_i(x,\bar{x};t,s,u,v)
\]
with
\begin{eqnarray*}
h_1(x,\bar{x} ;t,s,u,v) &=&f(\bar{x}-x+v-u) \bigl[Z_n(t-s,x-u)-Z_n(t-s,
\bar {x}-u)\bigr]
\\
&&{} \times\bigl[Z_n(t-s,x-v)-Z_n(t-s,\bar{x}-v)\bigr],
\\
h_2(x,\bar{x} ;t,s,u,v) &=&Df(v-u, x-\bar{x}) Z_n(t-s,x-u)
\\
&&{} \times\bigl[Z_n(t-s,x-v)-Z_n(t-s,\bar{x}-v)\bigr],
\\
h_3(x,\bar{x} ;t,s,u,v) &=&Df(v-u,\bar{x}-x) Z_n(t-s,
\bar{x}-v)
\\
&&{} \times\bigl[Z_n(t-s,x-u)-Z_n(t-s,\bar{x}-u)\bigr],
\\
h_4(x,\bar{x} ;t,s,u,v) &=&-D^2f(v-u,x-\bar{x})
Z_n(t-s,x-u)Z_n(t-s,x-v).
\end{eqnarray*}
Fix $p\in[2,\infty[$. It holds that
%
\begin{equation}
\label{s3.23} \E\bigl(\bigl |I_n(t,x,\bar{x})\bigr |^{{\trup{p}{2}}}\bigr) \le C
\sum_{i=1}^4 \E \bigl(\bigl |J_i^t(x,
\bar{x}) \bigr |^{{\trup{p}{2}}}\bigr).
\end{equation}
%
The next purpose is to obtain estimates for each term on the right
hand-side of \eqref{s3.23}. Let
\[
\mu_1(x,\bar{x})= \sup_{s\in[0,T]}\int
_{\R^3}\int_{\R^3} G(s,\mathrm{d}u)G(s,
\mathrm{d}v)f(\bar{x}-x+v-u).
\]
By Lemma 7.1 in \cite{h-hu-nu} and \eqref{cond-f} we have that  

$$\sup_{x,\bar{x}}\mu
_1(x,\bar{x})<\infty$$

Hence using firstly H\"older's inequality and then Cauchy--Schwarz's
inequality along with remark \ref{rs3.1}, we see that
%
\begin{eqnarray}
\label{s3.24} &&\E \bigl(\bigl |J_1^t(x,
\bar{x})\bigr |^{{\trup{p}{2}}} \bigr)
\nonumber
\\
&&\quad \le \biggl(\int_0^t \mathrm{d}s\int
_{\R^3}\int_{\R^3} G(s,\mathrm{d}u)G(s,
\mathrm{d}v) f(\bar{x}-x+v-u) \biggr)^{(\trup{p}{2})-1}
\nonumber
\\
&&\qquad{} \times\int_0^t \mathrm{d}s\int
_{\R^3}\int_{\R^3} G(s,\mathrm{d}u)G(s,
\mathrm{d}v)f(\bar {x}-x+v-u)
\nonumber
\\
&&\qquad\qquad{} \times \E \bigl( \bigl | \bigl[Z_n(t-s,x-u)-Z_n(t-s,
\bar{x}-u)\bigr]
\nonumber
\\[-8pt]
\nonumber
\\[-8pt]
&&\qquad\qquad\qquad{}\times\bigl[Z_n(t-s,x-v)-Z_n(t-s,\bar{x}-v)
\bigr] \bigr |^{{\trup{p}{2}}} \bigr)\nonumber
\nonumber
\\
&&\quad \le C \sup_{x,\bar x} \mu_1(x,\bar
x)^{\trup{p}{2}} \int_0^t \mathrm{d}s \sup
_{z_1-z_2=w} \E \bigl( \bigl | Z_n(t-s,z_1)-Z_n(t-s,
z_2) \bigr |^p \bigr)
\nonumber
\\
&&\quad \le C \int_0^t \mathrm{d}s \sup
_{z_1-z_2=w} \E \bigl( \bigl | Z_n(s,z_1)-Z_n(s,z_2) \bigr |^p \bigr).
\end{eqnarray}
Set
%
\begin{equation}
\label{s3.241} \mu_2(x,\bar{x})= \sup_{s\in[0,T]}\int
_{\R^3}\int_{\R^3} G(s,\mathrm{d}u)G(s,
\mathrm{d}v)\bigl |Df(v-u,x-\bar{x})\bigr |.
\end{equation}
The following property holds: there exists a positive finite constant
$C$ such that
\[
\mu_2(x,\bar{x})\le C|x-\bar{x}|^{\gamma}, \qquad
\gamma \in\,\bigl ]0,1].
\]
Indeed, this follows from Lemma 7.1 in \cite{h-hu-nu} and  {\bf (b)} in {\bf Hypothesis 1}:
\begin{eqnarray*}
\int_{\R^3}\int_{\R^3} G(s,\mathrm{d}u)G(s,
\mathrm{d}v)\bigl |Df(v-u,x-\bar{x})\bigr |&\le& C\int_{|z|\le 2T} \frac{|f(z+w)-f(z)|}{|z|}\\
& \le& C |w|^\gamma= C |x-\bar x|^\gamma
\end{eqnarray*}

Now, using the inequality $ab\le \tfrac{a^2+b^2}{2}$ we write
\begin{eqnarray}
\label{s3.25.0} \E\bigl(\bigl |J_2^t(x,\bar{x})\bigr |^{{\trup{p}{2}}}
\bigr) &\le& C \E \biggl(\int_0^t \mathrm{d}s\int
_{\R^3}\int_{\R^3} G(s,\mathrm{d}u)G(s,
\mathrm{d}v)\bigl |Df(v-u,x-\bar{x})\bigr | 
\nonumber
\\
&&{}\quad \times \biggl| Z_n(t-s,x-u)\bigl[Z_n(t-s,x-v)-Z_n(t-s,
\bar {x}-v)\bigr] \biggr| \biggr)^{\trup{p}{2}}
\nonumber
\\
&\le& C \E \biggl(\int_0^t \mathrm{d}s\int
_{\R^3}\int_{\R^3} G(s,\mathrm{d}u)G(s,
\mathrm{d}v)\bigl |Df(v-u,x-\bar{x})\bigr | 
\nonumber
\\
&&{}\quad \times |w|^\gamma \bigl|Z_n(t-s,x-u)\bigr|^2 \bigr)\biggr)^{\trup{p}{2}}
\nonumber
\\
&&+ C \E \biggl(\int_0^t \mathrm{d}s\int
_{\R^3}\int_{\R^3} G(s,\mathrm{d}u)G(s,
\mathrm{d}v)\bigl |Df(v-u,x-\bar{x})\bigr | 
\nonumber
\\
&&{}\quad \times |w|^{-\gamma} \bigl|Z_n(t-s,x-v)-Z_n(t-s,
\bar {x}-v) \bigr |^2 \bigr)\biggr)^{\trup{p}{2}}
\nonumber
\\
&=:& C\biggl( \E\Bigl[\bigl |J_{2,1}^t(x,\bar{x})\bigr |^{{\trup{p}{2}}}
\Bigr]+\E\bigl[\Bigl |J_{2,2}^t(x,\bar{x})\bigr |^{{\trup{p}{2}}}
\Bigr]\biggr).
\nonumber
\end{eqnarray}

Using H\"older's and  \eqref
{s3.21} along with the bound for  $\mu_2$, we have
%
\begin{eqnarray}
\label{s3.26.1} \E\bigl(\bigl |J_{2,1}^t(x,\bar{x})\bigr |^{{\trup{p}{2}}}
\bigr) &\le& C |w|^{\gamma p/2} \biggl(\int_0^t \mathrm{d}s\int
_{\R^3}\int_{\R^3} G(s,\mathrm{d}u)G(s,
\mathrm{d}v)\bigl |Df(v-u,x-\bar{x})\bigr | \biggr)^{\trup{p}{2}}
\nonumber
\\
&&{}\qquad \times\sup_n\sup_{(t,x)\in[0,T]\times \R^3}\E \bigl( \bigl | Z_n(t,x)\bigr |^p \bigr)\qquad\quad
\nonumber
\\
&\le& C |w|^{\gamma p/2}\biggl( \sup_{s\in [0,T]} \int
_{\R^3}\int_{\R^3} G(s,\mathrm{d}u)G(s,\mathrm{d}v)\bigl |Df(v-u,x-\bar{x})\bigr | \biggr)^{\trup{p}{2}}
\nonumber
\\
\nonumber
&\le& C |x-\bar{x}|^{\gamma p/2} |x-\bar{x}|^{\gamma p/2} = C |x-\bar{x}|^{\gamma p}
\end{eqnarray}
with $\gamma \in\,]0,1]$.

Using H\"older's inequality, Remark \ref{rs3.1} and  the bound for $\mu_2$, we have
%
\begin{eqnarray}
\label{s3.26.2} \E\bigl(\bigl |J_{2,2}^t(x,\bar{x})\bigr |^{{\trup{p}{2}}}
\bigr) &\le& C |w|^{-\gamma p/2} \biggl(\int_0^t \mathrm{d}s\int
_{\R^3}\int_{\R^3} G(s,\mathrm{d}u)G(s,
\mathrm{d}v)\bigl |Df(v-u,x-\bar{x})\bigr | \biggr)^{\tfrac{p}{2}-1}
\nonumber
\\
&&{}\qquad \times\int_0^t \mathrm{d}s \sup_{z_1-z_2=w}\E \bigl( \bigl | Z_n(t,z_1)-Z_n(t,z_2)\bigr |^p \bigr)\qquad\quad
\nonumber
\\
&&{}\qquad\qquad \times\sup_{s\in [0,T]} \int
_{\R^3}\int_{\R^3} G(s,\mathrm{d}u)G(s,
\mathrm{d}v)\bigl |Df(v-u,x-\bar{x})\bigr |\qquad\quad
\nonumber
\\
&\le& C |w|^{-\gamma p/2}\biggl( \sup_{s\in [0,T]} \int
_{\R^3}\int_{\R^3} G(s,\mathrm{d}u)G(s,\mathrm{d}v)\bigl |Df(v-u,x-\bar{x})\bigr | \biggr)^{\trup{p}{2}}
\nonumber
\\
&&{}\qquad \times\int_0^t \mathrm{d}s \sup_{z_1-z_2=w}\E \bigl( \bigl | Z_n(t,z_1)-Z_n(t,z_2)\bigr |^p \bigr)\qquad\quad
\nonumber
\\
\nonumber
&\le& C |w|^{-\gamma p/2} |w|^{\gamma p/2}
\int_0^t \mathrm{d}s \sup_{z_1-z_2=w}\E \bigl( \bigl | Z_n(t,z_1)-Z_n(t,z_2)\bigr |^p \bigr)\qquad\quad
\nonumber
\\
\nonumber
&\le& C \int_0^t \mathrm{d}s \sup_{z_1-z_2=w}\E \bigl( \bigl | Z_n(t,z_1)-Z_n(t,z_2)\bigr |^p \bigr)\qquad\quad
\end{eqnarray}

Thus, 
\begin{eqnarray}
\label{s3.25} \E\bigl(\bigl |J_2^t(x,\bar{x})\bigr |^{{\trup{p}{2}}}
\bigr) &\le& C \Bigl( |x-\bar{x}|^{\gamma p} + 
\int_0^t \mathrm{d}s \sup_{z_1-z_2=w}\E \bigl( \bigl | Z_n(t,z_1)-Z_n(t,z_2)\bigr |^p \bigr) \Bigr),
\end{eqnarray}

with $\gamma \in\,]0,1]$.

Similarly,
\begin{eqnarray}
\label{s3.26} \E\bigl(\bigl |J_3^t(x,\bar{x})\bigr |^{{\trup{p}{2}}}
\bigr) &\le& C \Bigl( |x-\bar{x}|^{\gamma p} + 
\int_0^t \mathrm{d}s \sup_{z_1-z_2=w}\E \bigl( \bigl | Z_n(t,z_1)-Z_n(t,z_2)\bigr |^p \bigr) \Bigr),
\end{eqnarray}

with $\gamma \in\,]0,1]$.

H\"older's and Cauchy--Schwarz's inequalities, along with \eqref{s3.21} and \eqref{H1-d}, imply
%
\begin{eqnarray}
\label{s3.27} \E\bigl(\bigl |J_4^t(x,\bar{x})\bigr |^{{\trup{p}{2}}}
\bigr) &\le& C \biggl(\int_0^t \mathrm{d}s\int
_{\R^3}\int_{\R^3} G(s,\mathrm{d}u)G(s,
\mathrm{d}v) \bigl |D^2f(v-u,x-\bar{x})\bigr | \biggr)^{(\trup{p}{2})-1}\quad
\nonumber
\\
&&{} \times\int_0^t \mathrm{d}s\int
_{\R^3}\int_{\R^3} G(s,\mathrm{d}u)G(s,
\mathrm{d}v)\bigl |D^2f(v-u,x-\bar{x})\bigr |
\nonumber
\\
&&{} \qquad\quad\times \E \bigl( \bigl |Z_n(t-s,x-u) Z_n(t-s,\bar{x}-v)\bigr |^{p/2} \bigr)
\nonumber
\\
&\le& C \biggl(\int_0^t \mathrm{d}s\int
_{\R^3}\int_{\R^3} G(s,\mathrm{d}u)G(s,
\mathrm{d}v) \bigl |D^2f(v-u,x-\bar{x})\bigr | \biggr)^{\trup{p}{2}}\quad
\nonumber
\\
&&{} \times\sup_n\sup_{(t,x)\in [0,T]\times \R^3} \E \bigl( \bigl | Z_n(t,x)\bigr |^p \bigr)
\nonumber
\\
&\le& C \biggl(\int_0^t \mathrm{d}s\int
_{|z|\le 2T} \frac{|f(z+w)-2f(z)-f(z-w)|}{|z|} ds \biggr)^{\trup{p}{2}}\quad
\nonumber
\\
&\le& C |x-\bar{x}|^{\gamma' p /2}.
\end{eqnarray}

From \eqref{s3.23}, \eqref{s3.24}, \eqref{s3.25}, \eqref{s3.26} and
\eqref{s3.27}, we obtain \eqref{s3.220}.
\end{proof}
%

For any $t\in[t_0,T]$, $x,\bar x\in K$, $p\in[1,\infty[$ set  $w=x-\bar x$ and
\begin{eqnarray*}
\varphi_{n,p}^0(t,x,\bar{x}) &=& \sup_{z_1-z_2=w} \E \bigl( \bigl |
X_n(t,z_1)-X_n(t,z_2) \bigr |^p
1_{L_n(t)} \bigr),
\\
\varphi_{n,p}^-(t,x,\bar{x}) &=& \sup_{z_1-z_2=w} \E \bigl( \bigl | X_n^-(t,z_1)-X_n^-(t,z_2
) \bigr |^p 1_{L_n(t)} \bigr),
\\
\varphi_{n,p}(t,x,\bar{x}) &=& \varphi_n^0(t,x,
\bar{x})+\varphi _n^-(t,x,\bar{x}).
\end{eqnarray*}

Proposition~\ref{pss3.1.1} is a consequence of the following assertion.
%
\begin{proposition}\label{pss3.1.2}
The hypotheses are the same as in Proposition~\ref{pss3.1.1}. Fix
$t\in[t_0,T]$, $x,\bar x\in K$. Then, for any $p\in[1,\infty[$,
$\rho\in\, ]0, \min(\gamma,\gamma_1,\gamma_2,\tfrac{\gamma'}{2}) [$,
%
\begin{equation}
\label{s3.28} \sup_{n\ge0}\varphi_{n,p}(t,x,\bar{x})
\le C|x-\bar x|^{\rho p}.
\end{equation}
\end{proposition}

The proof of this proposition relies on the next lemma and 
Gronwall's lemma.

%
\begin{lemma}\label{lss3.1.2}
We assume the same hypotheses as in Proposition~\ref{pss3.1.1}. For
any $n\ge1$, $t\in[t_0,T]$, $x,\bar x\in K$, $p\in[2,\infty[$,
there exists a finite constant $C$ (not depending on $n$) such that
%
\begin{eqnarray}
\label{s3.29} \varphi_{n,p}(t,x,\bar{x}) &\le& C
\biggl[f_n+|x-\bar x|^{\gamma p} + |x-\bar x|^{\trup{\gamma' p}{2}}
+|x-\bar{x}|^{\gamma_1 p } \nonumber\\
&& \qquad+|x-\bar{x}|^{\gamma_2 p }+ \int_0^t \mathrm{d}s \bigl(\varphi_{n,p}(s,x,\bar{x})\bigr)
 \biggr],
\end{eqnarray}
where $(f_n, n\ge1)$ is a sequence of positive real numbers which converges to
zero as $n\to\infty$, $\gamma,\gamma_1,\gamma_2\in]0,1]$, $\gamma'\in\,]0, 2]$.
\end{lemma}

We postpone the proof of this lemma to the end of this section.

\begin{proof} (of Proposition~\ref{pss3.1.2}).\\
We will prove this proposition by contradiction. Suppose that the Lemma~\ref{lss3.1.2} does not imply
the Proposition ~\ref{pss3.1.2}. It means that \eqref{s3.29} does not imply \eqref{s3.28}. In such case, there exists some
 $m\in \IN$ and $x,\bar x\in K$ such that \eqref{s3.29} is satisfied but \eqref{s3.28} (without the $sup_n$) does not. 
 But, since $f_n$ is a bounded sequence, there exists $C_0>0$ ($C_0$ depending only on $x,\bar x$) such that 
 $$
f_{m}\le C_0 |x-\bar x|^{\alpha p}, 
 $$
where $\alpha=\min(\gamma,\gamma_1,\gamma_2,\tfrac{\gamma'}{2})$. This inequality and \eqref{s3.29} leave us

\begin{eqnarray*}
\varphi_{m,p}(t,x,\bar{x}) &\le& C
\biggl[|x-\bar x|^{\gamma p} + |x-\bar x|^{\trup{\gamma' p}{2}}
+|x-\bar{x}|^{\gamma_1 p } \nonumber\\
&& \qquad+|x-\bar{x}|^{\gamma_2 p }+ \int_0^t \mathrm{d}s \varphi_{m,p}(s,x,\bar{x})
 \biggr].
\end{eqnarray*}

Then by Gronwall's lemma we have

\begin{equation}
\label{s3.30-0} \varphi_{m,p}(t,x,\bar{x}) \le C
\biggl[|x-\bar x|^{\gamma p} + |x-\bar x|^{\trup{\gamma' p}{2}}
+|x-\bar{x}|^{\gamma_1 p } 
+|x-\bar{x}|^{\gamma_2 p } \biggr].
\end{equation}

With the constant $C$ not depending on $m$.
We recall that $\gamma,\gamma_1,\gamma_2 \in\,]0,1]$ and $\gamma'
\in\,]0,2[$. Therefore, \eqref{s3.30-0} implies 
\begin{equation}
\label{s3.28.a} \varphi_{m,p}(t,x,\bar{x})
\le C|x-\bar x|^{\rho p},
\end{equation}
with $\rho\in\, ]0, \min(\gamma,\gamma_1,\gamma_2,\tfrac{\gamma'}{2}) [$.

This is a contradiction, thus \eqref{s3.29} does imply \eqref{s3.28}.
This ends the proof of Proposition~\ref{pss3.1.2}.

\end{proof}


\begin{proof}(of Lemma~\ref{lss3.1.2}).\\
Fix $p\in[2,\infty[$ and set $w=x-\bar x$. From \eqref{s3.6}, we have the following:\vspace*{-2pt}
\[
\varphi_{n,p}^0(t,x,\bar x):=\E \bigl(\bigl |X_n(t,x)-X_n(t,
\bar{x}) \bigr |^p 1_{L_n(t)} \bigr) \le C \sum
_{i=0}^6 R_n^i(t,x,\bar
x),
\]
with\vspace*{-2pt}
\begin{eqnarray*}
R_n^0(t,x,\bar x) &=&|X^0(t,x)-X^0(t,\bar x)|^p\\
R_n^1(t,x,\bar x) &=&\E \biggl(\biggl\llvert \int
_0^t \int_{\R^3}
\bigl[G(t-s,x-y)-G(t-s,\bar {x}-y)\bigr] A\bigl(X_n(s,y)\bigr) M(
\mathrm{d}s,\mathrm{d}y) \biggr\rrvert ^p1_{L_n(t)} \biggr),
\\
R_n^2(t,x,\bar x) &=&\E {\bigl(} \bigl |\bigl\langle\bigl[G(t-
\cdot,x-\ast)-G(t-\cdot,\bar{x}-\ast )\bigr] B\bigl(X_n^-(\cdot,\ast)
\bigr), w^n\bigr\rangle_{\mathcal{H}_t} \bigr |^p1_{L_n(t)}
{\bigr)},
\\
R_n^3(t,x,\bar x) &=&\E \bigl(\bigl\llvert \bigl\langle
\bigl[G(t-\cdot,x-\ast)-G(t-\cdot,\bar {x}-\ast)\bigr] \bigl[B(X_n)-B
\bigl(X_n^-\bigr)\bigr] (\cdot,\ast),w^n \bigr
\rangle_{\mathcal{H}_t}\bigr\rrvert ^p1_{L_n(t)} \bigr),
\\
R_n^4(t,x,\bar x) &=&\E \bigl(\bigl\llvert \bigl\langle
\bigl[G(t-\cdot,x-\ast)-G(t-\cdot,\bar {x}-\ast)\bigr] D\bigl(X_n(
\cdot,\ast)\bigr),h \bigr\rangle_{\mathcal{H}_t}\bigr\rrvert ^p1_{L_n(t)}
\bigr),
\\
R_n^5(t,x,\bar x) &=&\E \biggl(\biggl\llvert \int
_0^t \int_{\R^3}
\bigl[G(t-s,x-\mathrm{d}y)-G(t-s,\bar{x}-\mathrm{d}y)\bigr] b\bigl(X_n(s,y)
\bigr) \,\mathrm{d}s \biggr\rrvert ^p 1_{L_n(t)} \biggr).
\end{eqnarray*}

It is well know that with the  {\bf Hypothesis \ref{HypH1}.(a)} we have that (see for instance \cite[theorem 3.1]{h-hu-nu})  
\begin{eqnarray}
\label{s3.32.1} R_n^0(t,x,\bar{x})&\le& C_1|w|^{\gamma_1 p}+ C_2|w|^{\gamma_2 p}
\end{eqnarray}

for $\gamma_1,\gamma_2\in ]0,1]$. 

Using Burkholder's inequality and then Plancherel's identity, we have\vspace*{-2pt}
%
\begin{eqnarray}
\label{s3.32} R_n^1(t,x,\bar{x})&=&\E \biggl(\biggl
\llvert \int_0^t \int_{\R^3}
\bigl[G(t-s,x-y)-G(t-s,\bar{x}-y) \bigr] A\bigl(X_n(s,y)\bigr)M(
\mathrm{d}s,\mathrm{d}y) \biggr\rrvert ^p1_{L_n(t)} \biggr)
\nonumber
\\
&=&\E \biggl(\biggl\llvert \sum_{j\in\mathbb{N}}\int
_0^t \bigl\langle \bigl[G(t-s,x-\ast)\nonumber
\\[-1pt]
&&\phantom{\E \biggl(\biggl\llvert\sum_{j\in\mathbb{N}}\int
_0^t \bigl\langle \bigl[}{}-G(t-s,
\bar{x}-\ast) \bigr] A\bigl(X_n(s,\ast)\bigr), e_k(\ast)
\bigr\rangle_{\mathcal{H}}\, \mathrm{d}W_j(s) \biggr\rrvert
^p1_{L_n(t)} \biggr)
\nonumber
\\[-8pt]
\\[-8pt]
&\le& C\E \biggl( \biggl[\int_0^t \mathrm{d}s
\sum_{j\in\mathbb{N}} \bigl\llvert \bigl\langle \bigl[G(t-s,x-
\ast)\nonumber
\\[-1pt]
&&\phantom{C\E \biggl( \biggl[\int_0^t \mathrm{d}s
\sum_{j\in\mathbb{N}} \bigl\llvert \bigl\langle \bigl[}{}-G(t-s,\bar{x}-\ast) \bigr] A\bigl(X_n(s,\ast )\bigr),
e_k(\ast) \bigr\rangle_{\mathcal{H}} \bigr\rrvert
^21_{L_n(s)} \biggr] \biggr)^{{\trup{p}{2}}}
\nonumber
\\[-1pt]
&=& C\E \biggl(\biggl\llvert \int_0^t
\mathrm{d}s \bigl \| \bigl[G(t-s,x-\ast )-G(t-s,\bar{x}-\ast) \bigr] A
\bigl(X_n(s,\ast)\bigr)\bigr  \|_{\mathcal{H}}^2 \biggr\rrvert
1_{L_n(s)} \biggr)^{{\trup{p}{2}}}.
\nonumber
\end{eqnarray}
The process $\{Z_n(t,x):=A(X_n(t,x))1_{L_n(t)}, (t,x)\in[0,T]\times
\IR^3\}$ satisfies the assumption \eqref{s3.21}. Indeed, this is a
consequence of the linear growth of $A$ and \eqref{s4.18}. Then,\vadjust{\goodbreak} by
applying Lemma~\ref{lss3.1.1} and using the Lipschitz continuity of
$A$, we obtain
%
\begin{eqnarray}
\label{s3.33} R_n^1(t,x,\bar{x}) &\le& C \biggl\{ |x-
\bar{x}|^{\gamma p} + |x-\bar{x}|^{\gamma' p/2}  \nonumber
\\
&&\phantom{C\biggl\{} {} +\int_0^t \mathrm{d}s
\Bigl[ \sup_{z_1-z_2=w} \E \bigl( \bigl | X_n(s,z_1)-X_n(s,z_2) \bigr |^p 1_{L_n(s)} \bigr) \Bigr]\biggr
\}
\end{eqnarray}
with $\gamma \in\,]0,1] $ and $\gamma' \in\,]0,2]$.

For a given function $\rho: [0,T]\times\IR^3\to\IR$ and $t\in
[0,T]$, let $\tau_n$ be the operator defined by
%
\begin{equation}
\label{s3.34} \tau_n(\rho)=\rho \bigl(\bigl(s+2^{-n}
\bigr)\wedge t, x \bigr).
\end{equation}
Let $\mathcal{E}_n$ be the closed subspace of $\mathcal{H}_T$
generated by the orthonormal system of functions
\[
2^nT^{-1}1_{\Delta_i}(\cdot)\otimes
e_j(\ast),\qquad i=0,\ldots,2^n-1,\ j=1,\ldots,n,
\]
and denote by $\pi_n$ the orthogonal projection on $\mathcal{E}_n$.
Notice that $\pi_n\circ\tau_n$ is
a bounded operator on $\mathcal{H}_T$, uniformly in $n$.

Since $X_n^-(s,\ast)$ is $\mathcal{F}_{s_n}$-measurable, by using the
definition of $w^n$ we easily see that
\begin{eqnarray*}
R_n^2(t,x,\bar{x}) &=&\E \biggl(\biggl\llvert \int
_0^t \int_{\R^3}(
\pi_n\circ\tau_n) \bigl(\bigl[G(t-\cdot,x-\ast)-G(t-
\cdot,\bar{x}-\ast)\bigr]
\\
&&\phantom{\E\bigl(|} {}  \times B\bigl(X_n^-(
\cdot,\ast)\bigr) \bigr) (s,y) M(\mathrm{d}s,\mathrm{d}y) \biggr\rrvert
^p1_{L_n(s)} \biggr).
\end{eqnarray*}
By Burkholder's inequality and the properties of the operator $\pi_n\circ\tau_n$, this last expression is bounded up to a constant by
\[
\E \biggl(\int_0^t \mathrm{d}s \bigl\llVert
\bigl(\bigl[G(t-s,x-\ast)-G(t-s,\bar {x}-\ast)\bigr] B\bigl(X_n^-(s,
\ast)\bigr) \bigr) \bigr\rrVert _{\mathcal{H}}^2 1_{L_n(s)}
\biggr)^{{\trup{p}{2}}}.
\]
The properties of the function $B$ along with \eqref{s4.18} imply that
the process $\{Z_n(t,x):=B(X_n^-(t,x))1_{L_n(t)}, (t,x)\in[0,T]\times
\IR^3\}$ satisfies the hypotheses of Lemma~\ref{lss3.1.1}. This yields
%
\begin{eqnarray}
\label{s3.35} R_n^2(t,x,\bar{x}) &\le& C \biggl\{ |x-
\bar{x}|^{\gamma p} + |x-\bar{x}|^{\gamma' p/2}  \nonumber
\\
&&\phantom{C\biggl\{} {} +\int_0^t \mathrm{d}s
\Bigl[ \sup_{z_1-z_2=w} \E \bigl( \bigl | X_n^-(s,z_1)-X_n^-(s,z_2) \bigr |^p 1_{L_n(s)} \bigr) \Bigr]\biggr
\},
\end{eqnarray}

where as before, $\gamma \in\,]0,1] $ and $\gamma' \in\,]0,2]$.\vadjust{\goodbreak}

Cauchy--Schwarz's inequality along with \eqref{s3.101} yield
\begin{eqnarray*}
R_n^3(t,x,\bar{x}) &\le& C n^{\trup{3p}{2}}2^{n{\trup{p}{2}}}
\\
&&{} \times\E \biggl( \int_0^t \mathrm{d}s \bigl\|
\bigl[G(t-s,x-\ast)-G(t-s,\bar {x}-\ast)\bigr]
\\
&&\phantom{\times\E \bigl( \int_0^t \mathrm{d}s \bigl\|} {}   \times\bigl[B(X_n)-B
\bigl(X_n^-\bigr)\bigr](s,\ast)1_{L_n(s)} \bigr\|_{\mathcal
{H}}^2
\biggr)^{{\trup{p}{2}}}.
\end{eqnarray*}
Notice that an upper bound for the second factor on the right-hand side
of the preceding inequality could be obtained using
Lemma~\ref{lss3.1.1} with $Z_n(t,x):=[B(X_n(t,x))-B(X_n^-(t,x))]
1_{L_n}(t)$. However, this would not be a good strategy to compensate the
first factor (which explodes when $n\to\infty$). Instead, we will try
to quantify the discrepancy between $B(X_n(t,x))$ and $B(X_n^-(t,x))$.
This can be achieved by transferring the increments of the Green
function to increments of the process 
%
\begin{equation}
\label{s3.351} \hat{B}\bigl(X_n(t,x)\bigr)=\bigl[B
\bigl(X_n(t,x)\bigr)-B\bigl(X_n^-(t,x)\bigr)\bigr],
\end{equation}
in the same manner as we did in the proof of Lemma~\ref{lss3.1.1} (see
\cite{dss}, pages 19--20 for the original idea).

Indeed, similarly as in \eqref{s3.23}, we obtain
%
\begin{equation}
\label{s3.36} R_n^3(t,x,\bar{x}) \le C
n^{\trup{3p}{2}}2^{n{\trup{p}{2}}} \sum_{i=1}^4
\E \bigl(\bigl |K_i^t(x,\bar{x}) \bigr |^{{\trup{p}{2}}}1_{L_n(t)}
\bigr),
\end{equation}
where for any $i=1,\ldots,4$, $K_i^t(x,\bar x)$ is given by
$J_i^t(x,\bar x)$ of Lemma~\ref{lss3.1.1} with $Z_n$ replaced by $\hat
{B}(X_n)$.

%

With  the definition of $\hat{B}(X_n)$ given in
\eqref{s3.351}, we easily get
%
\begin{eqnarray}
\label{s3.37} && \E \bigl( \bigl | \hat{B}\bigl(X_n(s,x-y)\bigr)-\hat{B}
\bigl(X_n(s,\bar{x}-y)\bigr) \bigr |^p 1_{L_n(s)} \bigr)
\nonumber
\\
&&\quad \le C \bigl[\E \bigl( \bigl | X_n(s,x-y)-X_n^-(s,x-y)
\bigr |^p 1_{L_n(s)} \bigr)
\nonumber
\\[-8pt]
\\[-8pt]
&&\phantom{\quad \le C\bigl[} {}  +\E \bigl( \bigl | X_n(s,
\bar{x}-y)-X_n^-(s,\bar{x}-y) \bigr |^p 1_{L_n(s)} \bigr)
\bigr]
\nonumber
\\
&&\quad \le C n^{\trup{3p}{2}}2^{-np(\nu+1)/2},
\nonumber
\end{eqnarray}
uniformly in $(s,x,y)\in[0,T]\times\IR^3\times\IR^3$, where the
last bound is obtained by using \eqref{s4.19}.
This estimate will be applied to the study of the right-hand side of
\eqref{s3.36}.

For $i=1$, \eqref{s3.24} with $Z_n(s,y):=\hat B(X_n(s,y))1_{L_n(s)}$,
along with \eqref{s3.37} yields
%
\begin{equation}
\label{s3.38} \E\bigl(\bigl |K_1^t(x,\bar{x})\bigr |^{{\trup{p}{2}}}1_{L_n(t)}
\bigr) \le C n^{\trup
{3p}{2}}2^{-np(\nu+1)/2}.
\end{equation}

Let $\mu_2(x,\bar{x})$ be as in \eqref{s3.241}. Since $x,\bar{x}
\in K$, and $K$ is bounded,
\[
\sup_{x,\bar{x}\in K}\mu_2(x,\bar{x})\le C,
\]
for some finite constant $C>0$. Hence, \eqref{s3.25}, \eqref{s3.26}
(with the same choice of $Z_n$ as before) together with \eqref{s3.37} gives
%
\begin{equation}
\label{s3.39} \E\bigl(\bigl |K_2^t(x,\bar{x})\bigr |^{{\trup{p}{2}}}1_{L_n(t)}
\bigr) + \E \bigl(\bigl |K_3^t(x,\bar{x})\bigr |^{{\trup{p}{2}}}1_{L_n(t)}
\bigr) \le C n^{\trup{3p}{2}}2^{-np(\nu+1)/2}.
\end{equation}

Proceeding as in \eqref{s3.27}, but replacing $Z_n(s,y)$ by $\hat
B(X_n(s,y))1_{L_n(s)}$, we obtain
\[
\E\bigl(\bigl |K_4^t(x,\bar{x})\bigr |^{{\trup{p}{2}}}
1_{L_n(t)}\bigr) \le C |x-\bar{x}|^{\gamma'p/2} \int
_0^t \mathrm{d}s\sup_{y\in\R^3}
\E \bigl( \bigl | \hat{B}\bigl(X_n(s,y)\bigr) \bigr |^p1_{L_n(s)}
\bigr).
\]
By the definition of $\hat B(X_n)$, and applying \eqref{s4.19}, we have
\[
\sup_{(s,y)\in[0,T]\times\R^3} \E \bigl( \bigl | \hat {B}\bigl(X_n(s,y)
\bigr) \bigr |^p1_{L_n(s)} \bigr) \le C n^{\trup{3p}{2}}
2^{-np(\nu+1)/2}.
\]
Thus,
%
\begin{equation}
\label{s3.40} \E\bigl(\bigl |K_4^t(x,\bar{x})\bigr |^{{\trup{p}{2}}}
1_{L_n(t)}\bigr)\le C n^{\trup
{3p}{2}}2^{-np(\nu+1)/2}.
\end{equation}
Putting together \eqref{s3.36} and \eqref{s3.38}--\eqref{s3.40} yields
%
\begin{equation}
\label{s3.41} R_n^3(t,x,\bar{x}) \le C f_n,
\end{equation}
where $f_n:=n^{3p} 2^{-np [{{-np(\nu+1)/2}}-{\trup
{1}{2}} ]}=n^{3p} 2^{-np\nu/2}$. Since $\nu\in\,]0,1]$, $\lim_{n\to\infty}f_n=0$. Notice that each member of the sequence 
$\{f_n\}$ is positive.

To estimate the term
$R_n^4(t,x,\bar{x})$ se use  first
Cauchy--Schwarz's inequality and then, by applying Lemma~\ref{lss3.1.1}
with $Z_n$ replaced by $D(X_n) 1_{L_n}$.
The Lipschitz continuity of $D$ along with the estimate \eqref{s4.18}
ensure that assumption \eqref{s3.21} is satisfied. We obtain
%
\begin{eqnarray}
\label{s3.42} R_n^4(t,x,\bar{x}) &\le&\| h
\|_{\mathcal{H}_t}^p \E \bigl( \bigl |\bigl \|\bigl[G(t-\cdot,x-\ast)-G(t-
\cdot,\bar{x}-\ast)\bigr] D\bigl(X_n(\cdot,\ast)\bigr)
1_{L_n(t)} \bigr \|_{\mathcal{H}_t}^2\bigr  |^{{\trup{p}{2}}} \bigr)
\nonumber
\\
&\le& C \biggl\{ |x-
\bar{x}|^{\gamma p} + |x-\bar{x}|^{\gamma' p/2}  \nonumber
\\
&&\phantom{C\biggl\{} {} +\int_0^t \mathrm{d}s
\Bigl[ \sup_{z_1-z_2=w} \E \bigl( \bigl | X_n(s,z_1)-X_n(s,z_2) \bigr |^p 1_{L_n(s)} \bigr) \Bigr]\biggr
\},
\end{eqnarray}

where as before, $\gamma \in\,]0,1] $ and $\gamma' \in\,]0,2]$.\vadjust{\goodbreak}

After having applied the change of variable $u\mapsto x-\bar{x}+y$, we have
\[
R_n^5(t,x,\bar{x}) =\E \biggl(\biggl\llvert \int
_0^t \int_{\R^3}G(t-s,x-
\mathrm{d}y)\bigl[b\bigl(X_n(s,y)\bigr) -b\bigl(X_n(s,y-x+
\bar{x})\bigr) \bigr]\,\mathrm{d}s \biggr\rrvert ^p1_{L_n(t)}
\biggr).
\]
Applying H\"older's inequality, we obtain
%
\begin{eqnarray}
\label{s3.31}
&&R_n^5(t,x,\bar{x}) \nonumber
\\
&&\quad\le  \biggl(\int
_0^t\int_{\R^3} G(t-s,x-
\mathrm{d}y)\,\mathrm{d}s \biggr)^{p-1}
\nonumber
\\
&&\qquad  {} \times\int_0^t \int_{\R^3}
G(t-s,x-\mathrm{d}y)\E \bigl( \bigl |b\bigl(X_n(s,y)\bigr)-b
\bigl(X_n(s,y-x+\bar{x})\bigr) \bigr |^p 1_{L_n(s)}
\bigr)\,\mathrm{d}s\nonumber
\\
&&\quad \le  C\int_0^t \mathrm{d}s \sup
_{y\in\R^3}\E \bigl( \bigl |X_n(s,x-y)-X_n(s,
\bar{x}-y) \bigr |^p1_{L_n(s)} \bigr).
\nonumber
\\
&&\quad \le  C\int_0^t \mathrm{d}s \sup
_{z_1-z_2=w}\E \bigl( \bigl |X_n(s,z_1)-X_n(s,z_2) \bigr |^p1_{L_n(s)} \bigr)
\end{eqnarray}

Bringing together the inequalities \eqref{s3.32.1}, \eqref{s3.33}, \eqref{s3.35},
\eqref{s3.41}, \eqref{s3.42} and \eqref{s3.31}, yields
\begin{eqnarray}
\label{s3.41.1}
&& \sup_{z_1-z_2=w} \E \bigl(\bigl |X_n(t,z_1)-X_n(t,z_2) \bigr |^p
1_{L_n(t)} \bigr)
\nonumber
\\
&& \quad \le C \biggl\{ f_n+ |x-\bar{x}|^{\gamma p } +|x-\bar{x}|^{\gamma' \trup{p}{2} }
+|x-\bar{x}|^{\gamma_1 p } +|x-\bar{x}|^{\gamma_2 p }
\nonumber
\\
&&\phantom{\quad \le C\biggl\{} {}  + \int
_0^t \mathrm{d}s \Bigl[ \sup_{z_1-z_2=w} \E \bigl( \bigl | X_n(s,z_1)-X_n(s,z_2)\bigr  |^p 1_{L_n(s)} \bigr) \Bigr]
\nonumber
\\
&&\phantom{\quad \le C\biggl\{} {}  + \int_0^t
\mathrm{d}s \Bigl[ \sup_{z_1-z_2=w} \E \bigl( \bigl |X_n^-(s,z_1)-X_n^-(s,z_2) \bigr |^p 1_{L_n(s)} \bigr) \Bigr] \biggr
\}.
\end{eqnarray}
With this, we see that $\varphi_{n,p}^0(t,x,\bar x)$ is bounded by the
right-hand side of \eqref{s3.41.1}.

Finally, we prove that the same bound holds for $\varphi
_{n,p}^-(t,x,\bar x)$ too. Indeed,
For every $i=1,\ldots,5$, we consider the terms $R_n^i(t,x,\bar x)$
defined in the first part of the proof,
and we replace the domain of integration of the time variable $s$
($[0,t]$) by $[0,t_n]$. We denote the corresponding new expressions by
$S_n^i(t,x,\bar x)$.
From \eqref{s3.8.2}, we obtain the following
\[
\varphi_{n,p}^-(t,x,\bar x) \le C \sum_{i=1}^5
S_n^i(t,x,\bar x).
\]
Since $t_n\le t$, it can be checked that, similarly as for
$R_n^i(t,x,\bar x)$,
$S_n^i(t,x,\bar x)$, $i=1,\ldots,5$, are bounded by \eqref{s3.33},
\eqref{s3.35}, \eqref{s3.41}, \eqref{s3.42}, \eqref{s3.31},
respectively. This ends the
proof of the lemma.
\end{proof}

\subsection{Increments in time}
\label{ss3.2}

Throughout this section, we fix $t_0\in\,]0,T]$, and a compact set
$K\subset\IR^3$. We shall prove the following proposition.
%
\begin{proposition}
\label{pss3.2.1}
Assume that {\bf Hypothesis 1} and {\bf Hypothesis 2} holds. Fix $t, \bar t\in[t_0,T]$ and 
set $\kappa \in\, \Big]0,\min\Big(\gamma_1,\gamma_2,\gamma,\tfrac{\gamma'}{2}\Big) \Big[$.
Then for any $p\in[1,\infty)$ there exists a finite constant $C$ such that
%
\begin{equation}
\label{s3.43} \sup_{n\ge1} \sup_{x\in K}\bigl \Vert
\bigl(X_n(t,x) - X_n(\bar t,x)\bigr) 1_{L_n(\bar t)}
\bigr \Vert_p \le C|t-\bar t|^\rho.
\end{equation}
with $\rho \in\, \Big]0,\min\Big(\gamma_1,\gamma_2,\gamma,\tfrac{\gamma'}{2},\tfrac{\nu+1}{2},
\tfrac{\rho_1+\kappa}{2},\tfrac{\rho_2}{2}\Big) \Big[$

\end{proposition}

Notice that the parameter $\kappa$ is obtained from the H\"older continuity in space variable for the process $X_n$
and from this reason 
belongs to the interval $\Big]0,\min\Big(\gamma_1,\gamma_2,\gamma,\tfrac{\gamma'}{2}\Big) \Big[$ (see Proposition 
\ref{pss3.1.1}).

The next lemma is meant to play a similar r\^ole than Lemma~\ref
{lss3.1.1} but in this case, for integrals containing increments in
time of the Green function $G(t)$.

\begin{lemma}
\label{lss3.2.1}
Consider a sequence of stochastic processes $\{D_n(t,x), (t,x)\in
[0,T]\times\IR^3\}$, $n\ge1$, satisfying the following conditions:

For any $p\in[2,\infty[$,
%
\begin{equation}
\label{s3.44} \sup_n\sup_{(t,x)\in[t_0,T]\times\R^3} \E \bigl(
\bigl |D_n(t,x)\bigr  |^p \bigr)\le C.
\end{equation}
There exists $\kappa>0$ and for any $x,y\in K$,
%
\begin{equation}
\label{s3.45} \sup_n \sup_{t\in[t_0,T]} \E
\bigl( \bigl |D_n(t,x)-D_n(t,y) \bigr |^p \bigr)\le
C|x-y|^{\kappa p},
\end{equation}
where $C$ is a finite constant and $\rho>0$. Suppose in adition that hypotheses \ref{HypH1} and \ref{HypH2} 
(see page \pageref{HypH1}) are satisfied.

For $0\le t_0\le t\le\bar{t} \le T$ and $x\in K$, set
\[
J_n(t,\bar t,x)=\int_0^t
\mathrm{d}s \bigl \Vert D_n(t,\ast) \bigl[G(\bar t-s,x-\ast )-G(t-s,x-\ast)
\bigr]\bigr \Vert^2_{\mathcal{H}}.
\]
Then, for any $p\in[2,\infty[$ there exists a finite constant $C>0$
such that
%
\begin{equation}
\label{s3.47} \E \bigl(J_n(t,\bar t,x)^{\trup{p}{2}} \bigr) \le C
\bigl( |\bar{t}-t|^{\gamma p} \bigr),
\end{equation}
with $0<\gamma<\min\bigl(\kappa,\tfrac{\nu+1}{2},\tfrac{\rho_1+\kappa}{2},\tfrac{\rho_2}{2}\bigr)$.
\end{lemma}

\begin{proof}
First of all we notice that, as a consequence of Burkholder's
inequality, the $L^p$-moment of the stochastic integral
\[
\int_0^t \int_{\R^3}
D_n(t,y)\bigl[G(\bar t-s,x-y)-G(t-s,x-y)\bigr] M(\mathrm{d}s,
\mathrm{d}y),
\]
is bounded up to a positive constant, by $\E (J_n(t,\bar
t,x)^{\trup{p}{2}} )$.

We write $J_n(t,\bar t,x)$ using \eqref{fundamental}. This gives
\begin{eqnarray*}
J_n(t,\bar t,x)&=& C\int_0^t
\mathrm{d}s \int_{\R^3}\int_{\R^3}
D_n(t,y)\bigl[G(\bar t-s,x-dy)-G(t-s,x-dy)\bigr]
\\
&&{} \times D_n(t,z)\bigl[G(\bar t-s,x-dz)-G(t-s,x-dz)\bigr]
f(y-z)\\
&=& C\int_0^t
\mathrm{d}s \int_{\R^3}\int_{\R^3}
D_n(t-s,x-y)\bigl[G(\bar t-t+s,dy)-G(s,dy)\bigr]
\\
&&{} \times D_n(t-s,x-z)\bigl[G(\bar t-t+s,dz)-G(s,dz)\bigr]
f(y-z)
\end{eqnarray*}
As was noted in \cite{h-hu-nu} the integral with respect to the space variables $y$ and $z$ is taken in the 
sphere $S^2$ in the three dimensional space, this is because of the structure of the fundamental solution $G$. We denote by
$\xi=\tfrac{y}{|y|}$ and $\eta=\tfrac{z}{|z|}$, moreover we denote by $\sigma(d\xi)$ and $\sigma(d\eta)$ the uniform measure
on $S^2$, so 
\begin{eqnarray*}
 G(s,dy)&=& \tfrac{1}{4\pi} \sigma(d\xi)\\
 G(s,dz)&=& \tfrac{1}{4\pi} \sigma(d\eta)
 \end{eqnarray*}

Denote by $h=\bar t - t$.
Then, we do the following decomposition
%
\begin{equation}
\label{s3.48} \E \bigl(J_n(t,\bar t,x)^{\trup{p}{2}} \bigr)\le C
\sum_{k=1}^4 \E \bigl(\bigl |Q^i(t,
\bar{t},x)\bigr |^{{\trup{p}{2}}} \bigr),
\end{equation}
where 
%
\begin{eqnarray*}
 Q^1(t,\bar{t},x)&:=&\int_0^t
\mathrm{d}s\int_{S^2} \int_{S^2} \bigl(s+h\bigr)^2 f\bigl([s+h]\xi-[s+h]\eta\bigl)\\
&&{}\times \biggl[D_n \big(t-s,x-(s+h)\xi \big)-D_n(t-s,x-s\xi) \biggr]\\
&&{}\times \biggl[D_n \biggl(t-s,x-(s+h)\eta\biggr)-D_n(t-s,x-s\eta) \biggr] \sigma(d\xi)\sigma(d\eta)\\
 Q^2(t,\bar{t},x)&:=&\int_0^t
\mathrm{d}s\int_{S^2} \int_{S^2}\Bigl[ \bigl(s+h\bigr)^2 f\Bigl([s+h]\xi-[s+h]\eta\Bigl)-
s(s+h) f\bigl(s\xi-[s+h]\eta\bigl)\Bigr]\\
&&{}\times D_n(t-s,x-s\xi)  \biggl[D_n \biggl(t-s,x-(s+h)\eta\biggr)-D_n(t-s,x-s\eta) \biggr] \sigma(d\xi)\sigma(d\eta)\\
 Q^3(t,\bar{t},x)&:=&\int_0^t
\mathrm{d}s\int_{S^2} \int_{S^2}\Bigl[ \bigl(s+h\bigr)^2 f\Bigl([s+h]\xi-[s+h]\eta\Bigl)-
s(s+h) f\bigl([s+h]\xi-s\eta\bigl)\Bigr]\\
&&{}\times D_n(t-s,x-s\eta)  \biggl[D_n \biggl(t-s,x-(s+h)\xi\biggr)-D_n(t-s,x-s\xi) \biggr] \sigma(d\xi)\sigma(d\eta)\\
 Q^4(t,\bar{t},x)&:=&\int_0^t
\mathrm{d}s\int_{S^2} \int_{S^2}\Bigl[ \bigl(s+h\bigr)^2 f\Bigl([s+h]\xi-[s+h]\eta\Bigl)-
s(s+h) f\bigl(s\xi-[s+h]\eta\bigl)\\
&&\hspace*{3cm} -s(s+h) f\bigl([s+h]\xi-s\eta\bigl) + s^2 f\bigl(s\xi-s\eta\bigl)
\Bigr]\\
&&{}\times D_n(t-s,x-s\eta) D_n(t-s,x-s\xi) \biggr] \sigma(d\xi)\sigma(d\eta)
\end{eqnarray*}%

Following the arguments of the proof of Theorem 4.1 in \cite{h-hu-nu}, we see that by using 
H\"older inequality, Cauchy-Schwarz inequality, hypothesis on $D$, Lemma 7.1 in  \cite{h-hu-nu} and condition \eqref{cond-f} we obtain 
%
\begin{eqnarray}
\label{s3.49} 
 \E\bigr(|Q^1(t,\bar{t},x)|^{p/2}\bigr)&\le& C \left(\int_0^t
\mathrm{d}s\int_{S^2} \int_{S^2} \bigl(s+h\bigr)^2 f\bigl([s+h]\xi-[s+h]\eta\bigl) \sigma(d\xi)\sigma(d\eta) \right)^{\tfrac{p}{2}-1}
\nonumber
\\
&&{}\times \int_0^t
\mathrm{d}s\int_{S^2} \int_{S^2} \bigl(s+h\bigr)^2 f\bigl([s+h]\xi-[s+h]\eta\bigl) \sigma(d\xi)\sigma(d\eta)
\nonumber
\\
&&{}\hspace*{1.5cm} \times \E\Biggl| \biggl[D_n \big(t-s,x-(s+h)\xi \big)-D_n(t-s,x-s\xi) \biggr]
\nonumber
\\
&&{}\hspace*{2.5cm}\times \biggl[D_n \biggl(t-s,x-(s+h)\eta\biggr)-D_n(t-s,x-s\eta) \biggr]\Biggr|^{p/2} 
\label{s3.52-3} 
\\
&\le& C \left(\int_0^t
\mathrm{d}s\int_{S^2} \int_{S^2} \bigl(s+h\bigr)^2 f\bigl([s+h]\xi-[s+h]\eta\bigl) \sigma(d\xi)\sigma(d\eta) \right)^{\tfrac{p}{2}-1}
\nonumber
\\
&&{}\times \int_0^t
\mathrm{d}s\int_{S^2} \int_{S^2} \bigl(s+h\bigr)^2 f\bigl([s+h]\xi-[s+h]\eta\bigl) \sigma(d\xi)\sigma(d\eta)
\nonumber
\\
&&{}\hspace*{1.5cm} \times \biggl[\E\Biggl| D_n \big(t-s,x-(s+h)\xi \big)-D_n(t-s,x-s\xi)\Biggr|^{p} \biggr]^{1/2}
\nonumber
\\
&&{}\hspace*{2.5cm}\times \biggl[\E\Biggl| D_n \biggl(t-s,x-(s+h)\eta\biggr)-D_n(t-s,x-s\eta) \Biggr|^{p} \biggr]^{1/2} 
\nonumber
\\
&\le& C |h|^{p\kappa} |\eta|^{p\kappa/2}|\xi|^{p\kappa/2}\left(\int_0^t
\mathrm{d}s\int_{S^2} \int_{S^2} \bigl(s+h\bigr)^2 f\bigl([s+h]\xi-[s+h]\eta\bigl) \sigma(d\xi)\sigma(d\eta) \right)^{\tfrac{p}{2}}
\nonumber
\\
&=& C |h|^{p\kappa} \left(\int_0^t \mathrm{d}s \int_{\R^3}\int_{\R^3}f(y-z)G(s+h,dy)G(s+h,dz) \right)^{\tfrac{p}{2}}
\nonumber
\\
&\le& C |h|^{p\kappa} \left( \int_0^t \mathrm{d}s \int_{|z|\le 2(s+h)} \frac{f(z)}{|z|} dz\right)^{\tfrac{p}{2}} 
\nonumber
\\
&\le& C |h|^{p\kappa} = C |\bar t- t|^{p\kappa}
\end{eqnarray}%

For $Q^2$, by making the trivial decomposition $(s+h)^2=s(s+h)+h(s+h)$ we get

\begin{eqnarray}\label{s3.52-a} 
\E \bigl( \bigl |Q^2(t,\bar{t},x) \bigr |^{{\trup{p}{2}}} \bigr) &\le& C \Bigg(\int_0^t
\mathrm{d}s\int_{S^2} \int_{S^2}s\bigl(s+h\bigr) \Bigl|f\Bigl([s+h]\xi-[s+h]\eta\Bigl)-
 f\bigl(s\xi-[s+h]\eta\bigl)\Bigr| \sigma(d\xi)\sigma(d\eta)
\nonumber
\\
&&{}\hspace*{1cm}\times |D_n(t-s,x-s\xi)|  \biggl|D_n \biggl(t-s,x-(s+h)\eta\biggr)-D_n(t-s,x-s\eta) \biggr|  \Bigg)^{\tfrac{p}{2}}
\nonumber
\\
&& +  C \Bigg(\int_0^t
\mathrm{d}s\int_{S^2} \int_{S^2}  h \bigl(s+h\bigr) f\Bigl([s+h]\xi-[s+h]\eta\Bigl) \sigma(d\xi)\sigma(d\eta)
\nonumber
\\
&&{}\hspace*{1cm}\times |D_n(t-s,x-s\xi)|  \biggl|D_n \biggl(t-s,x-(s+h)\eta\biggr)-D_n(t-s,x-s\eta) \biggr|  \Bigg)^{\tfrac{p}{2}}
\nonumber
\\
&=:& Q^{2,1} +  Q^{2,2}
\end{eqnarray}
For $ Q^{2,1}$ we use H\"older inequality, Cauchy-Schwarz inequality, assumptions \eqref{s3.44} and \eqref{s3.45}, condition \eqref{H2-c1} along 
with the change of variables $\eta\mapsto -\eta$ and we obtain

\begin{eqnarray}\label{s3.52-1} 
Q^{2,1} &\le& C \Bigg(\int_0^t
\mathrm{d}s\int_{S^2} \int_{S^2}s\bigl(s+h\bigr) \Bigl|f\Bigl([s+h]\xi-[s+h]\eta\Bigl)-
 f\bigl(s\xi-[s+h]\eta\bigl)\Bigr| \sigma(d\xi)\sigma(d\eta) \Bigg)^{\tfrac{p}{2}-1}
\nonumber
\\
&&{}\times\int_0^t
\mathrm{d}s\int_{S^2} \int_{S^2}s\bigl(s+h\bigr) \Bigl|f\Bigl([s+h]\xi-[s+h]\eta\Bigl)-
 f\bigl(s\xi-[s+h]\eta\bigl)\Bigr| \sigma(d\xi)\sigma(d\eta)
 \nonumber
 \\
 &&{}\hspace*{1cm}\times \Big(\E||D_n(t-s,x-s\xi)|^{p}\Big)^{1/2} \Bigg[\E \biggl|D_n \biggl(t-s,x-(s+h)\eta\biggr)-
 D_n(t-s,x-s\eta) \biggr|^p  \Bigg]^{\tfrac{1}{2}}
\nonumber
\\
&\le& C |h|^{p\kappa/2} |\eta|^{p\kappa/2}\Bigg(\int_0^t
\mathrm{d}s\int_{S^2} \int_{S^2}s\bigl(s+h\bigr) \Bigl|f\Bigl([s+h]\xi-[s+h]\eta\Bigl)-
 f\bigl(s\xi-[s+h]\eta\bigl)\Bigr| \sigma(d\xi)\sigma(d\eta) \Bigg)^{\tfrac{p}{2}}
\nonumber
\\
&\le& C |h|^{p\kappa/2} |h|^{p\rho_1/2} = C |h|^{\tfrac{p\kappa+p\rho_1}{2}}  
\end{eqnarray}

with $\kappa,\rho_1\in\,]0,1]$.

For $ Q^{2,2}$ we use H\"older inequality, Cauchy-Schwarz inequality, assumptions \eqref{s3.44} and 
\eqref{s3.45}, condition \eqref{H2-c1} and we obtain

\begin{eqnarray}
Q^{2,2} &\le& C |h|^{p/2}\Bigg(\int_0^t
\mathrm{d}s\int_{S^2} \int_{S^2} \bigl(s+h\bigr) f\Bigl([s+h]\xi-[s+h]\eta\Bigl) \sigma(d\xi)\sigma(d\eta) \Bigg)^{\tfrac{p}{2}-1}
\nonumber
\\
&&{}\times\int_0^t
\mathrm{d}s\int_{S^2} \int_{S^2} \bigl(s+h\bigr) f\Bigl([s+h]\xi-[s+h]\eta\Bigl) \sigma(d\xi)\sigma(d\eta)
 \nonumber
 \\
 &&{}\hspace*{1cm}\times \Big(\E||D_n(t-s,x-s\xi)|^{p}\Big)^{1/2} \Bigg[ \E\biggl|D_n \biggl(t-s,x-(s+h)\eta\biggr)-
 D_n(t-s,x-s\eta) \biggr|  \Bigg]^{\tfrac{1}{2}}
\nonumber
\\
&\le& C |h|^{p/2} |h|^{p\kappa/2} |\eta|^{p\kappa/2}\Bigg(\int_0^t
\mathrm{d}s\int_{S^2} \int_{S^2} \bigl(s+h\bigr) f\Bigl([s+h]\xi-[s+h]\eta\Bigl)\sigma(d\xi)\sigma(d\eta) \Bigg)^{\tfrac{p}{2}}
\nonumber
\\
&\le& C |h|^{\frac{p+p\kappa}{2}} \Bigg(\int_0^t
\frac{\mathrm{d}s}{s+h}\int_{\R^3} \int_{\R^3} f(y-z)G(s+h,dy)G(s+h,dz) \Bigg)^{\tfrac{p}{2}}
\nonumber
\\
&\le& C |h|^{\frac{p+p\kappa}{2}} \Bigg(\int_0^t
\frac{\mathrm{d}s}{s+h}\int_{|z|\le 2(s+h)}  \frac{f(z)}{|z|}dz \Bigg)^{\tfrac{p}{2}}
\nonumber
\\
&\le& C |h|^{\frac{p+p\kappa}{2}}  \label{s3.52-2} 
\end{eqnarray}

with $\kappa,\rho_1\in\,]0,1]$. Hence, by \eqref{s3.52-1} and \eqref{s3.52-2} we set

\begin{eqnarray}\label{s3.52} 
\E \bigl( \bigl |Q^2(t,\bar{t},x) \bigr |^{{\trup{p}{2}}} \bigr) &\le& C |h|^{\tfrac{p\kappa+p\rho_1}{2}}  
\end{eqnarray}

with $\kappa,\rho_1\in\,]0,1]$.

Similarly,
%
\begin{equation}
\label{s3.53} \E \bigl( \bigl |Q^3(t,\bar{t},x) \bigr |^{{\trup{p}{2}}} \bigr) \le C |h|^{\tfrac{p\kappa+p\rho_1}{2}}  
,
\end{equation}
where $\kappa,\rho_1\in\,]0,1]$.

By applying H\"older's and Cauchy--Schwarz's inequalities along with
\eqref{s3.44}, we get
%
\begin{eqnarray*}
 \E\left(|Q^4(t,\bar{t},x)|^{p/2}\right) &\le& C \Biggl(\int_0^t
\mathrm{d}s\int_{S^2} \int_{S^2}\Bigl| \bigl(s+h\bigr)^2 f\Bigl([s+h]\xi-[s+h]\eta\Bigl)-
s(s+h) f\bigl(s\xi-[s+h]\eta\bigl)\\
&&\hspace*{1.5cm} -s(s+h) f\bigl([s+h]\xi-s\eta\bigl) + s^2 f\bigl(s\xi-s\eta\bigl)
\Bigr| \sigma(d\xi)\sigma(d\eta)  \Biggr)^{\tfrac{p}{2}-1}\\
&&{}\times \int_0^t
\mathrm{d}s\int_{S^2} \int_{S^2}\Bigl| \bigl(s+h\bigr)^2 f\Bigl([s+h]\xi-[s+h]\eta\Bigl)-
s(s+h) f\bigl(s\xi-[s+h]\eta\bigl)\\
&&\hspace*{1.5cm} -s(s+h) f\bigl([s+h]\xi-s\eta\bigl) + s^2 f\bigl(s\xi-s\eta\bigl)
\Bigr| \sigma(d\xi)\sigma(d\eta) \\
&&{}\times \bigg[\E|D_n(t-s,x-s\eta)|^p\bigg]^{1/2} \bigg[\E|D_n(t-s,x-s\xi)|^p \bigg]^{1/2} 
\\ &\le& C \Biggl(\int_0^t
\mathrm{d}s\int_{S^2} \int_{S^2}\Bigl| \bigl(s+h\bigr)^2 f\Bigl([s+h]\xi-[s+h]\eta\Bigl)-
s(s+h) f\bigl(s\xi-[s+h]\eta\bigl)\\
&&\hspace*{1.5cm} -s(s+h) f\bigl([s+h]\xi-s\eta\bigl) + s^2 f\bigl(s\xi-s\eta\bigl)
\Bigr| \sigma(d\xi)\sigma(d\eta)  \Biggr)^{\tfrac{p}{2}}
\\
&\le& C\Big(Q^{4,1}+Q^{4,2}+Q^{4,3}+Q^{4,4}  \Big)
\end{eqnarray*}%
Where $Q^{4,i}$, for $i=1,\ldots,4$, is defined exactly as $R_4^i$ in \cite[page 20]{h-hu-nu}. Then, by following 
exactly the same procedure as in \cite{h-hu-nu} we arrive to

\begin{eqnarray}\label{s3.55}
 \E\left(|Q^4(t,\bar{t},x)|^{p/2}\right) &\le& C \Bigl(|h|^{\tfrac{p\rho_2}{2}} +|h|^{p\tfrac{\rho_1+1}{2}} +
 |h|^{p\tfrac{\nu+1}{2}} +|h|^{p(1-\epsilon)} \Bigr)
\end{eqnarray}%

%
with  $\nu,\rho_1\in\,]0,1]$ and $\rho_2\in\,]0,2]$ and where the last inequality is true for any $\epsilon>0$.

The inequalities \eqref{s3.48}, \eqref{s3.52}, \eqref{s3.53}, \eqref
{s3.55} imply \eqref{s3.47} which completes the proof.
\end{proof}
%

\begin{proof}(of Proposition~\ref{pss3.2.1}).\\
Fix $0\le t\le\bar{t}\le T$, $x\in K$, $p\in[2,\infty[$, and
according to \eqref{s3.6} consider the decomposition
\[
\E \bigl(\bigl |X_n(\bar{t},x)-X_n(t,x) \bigr |^p
1_{L_n(\bar{t})} \bigr) \le C \sum_{i=0}^6
R_n^i(t,\bar{t},x),
\]
where
\begin{eqnarray*}
R_n^0(t,\bar{t},x) &=&\bigl| X^0(t,x)-X^0(\bar t,x)\bigr |^p
\\
R_n^1(t,\bar{t},x) &=&\E \biggl(\biggl\llvert \int
_0^{\bar{t}}\int_{\R^3} \bigl[G(
\bar {t}-s,x-y)-G(t-s,x-y) \bigr]
\\
&&\phantom{\E\bigl(|} {}   \times A\bigl(X_n(s,y)\bigr)
M(\mathrm{d}s,\mathrm{d}y) \biggr\rrvert ^p1_{L_n(\bar
{t})} \biggr),
\\
R_n^2(t,\bar{t},x) &=&\E {\bigl(} \bigl | \bigl\langle
\bigl[G(\bar{t}-\cdot,x-\ast )-G(t-\cdot,x-\ast) \bigr] B\bigl(X_n^-(
\cdot,\ast)\bigr),w^n \bigr\rangle _{\mathcal{H}_{\bar t}}
\bigr |^p1_{L_n(\bar{t})} {\bigr)},
\\
R_n^3(t,\bar{t},x) &=&\E \bigl( \bigl | \bigl\langle \bigl[G(\bar{t}-
\cdot,x-\ast )-G(t-\cdot,x-\ast) \bigr]
\\
&&\phantom{\E\bigl(|\bigl\langle} {}  \times\bigl[B(X_n)-B
\bigl(X_n^-\bigr)\bigr](\cdot,\ast),w^n \bigr\rangle
_{\mathcal{H}_{\bar t}} \bigr |^p1_{L_n(\bar{t})} \bigr),
\\
R_n^4(t,\bar{t},x) &=&\E \bigl( \bigl | \bigl\langle \bigl[G(
\bar{t}-\cdot,x-\ast )-G(t-\cdot,x-\ast) \bigr] D\bigl(X_n(\cdot,\ast)
\bigr),h \bigr\rangle _{\mathcal{H}_{\bar t}} \bigr |^p1_{L_n(\bar{t})} \bigr),
\\
R_n^5(t,\bar{t},x) &=&\E \biggl(\biggl\llvert \int
_0^{\bar{t}} \int_{\R^3} \bigl[G(
\bar {t}-s,x-\mathrm{d}y)-G(t-s,x-\mathrm{d}y) \bigr] b\bigl(X_n(s,y)
\bigr) \,\mathrm{d}s \biggr\rrvert ^p1_{L_n(\bar
{t})} \biggr).
\end{eqnarray*}

Let $\gamma'=\min(\gamma_1,\gamma_2)$. By the assumptions on $\bigtriangleup v_0$ and $\bar v_0$ 
and by Lemma 4.9 in \cite{dss} we have
\begin{eqnarray}
\label{s3.570.1}
 R_n^0(t,\bar{t},x)\le C |t-\bar t|^{p\gamma'}
\end{eqnarray}

Similarly as for the term $R_n^1(t,x,\bar x)$ in the proof of Lemma~\ref{lss3.1.2} (see \eqref{s3.32}), we have
%
\begin{equation}
\label{s3.571}
 R_n^1(t,\bar{t},x) \le C\E \biggl(\int
_0^{\bar t} \mathrm{d}s \bigl \| \bigl[G(\bar t-s,x-\ast
)-G(t-s,x-\ast) \bigr] A\bigl(X_n(s,\ast)\bigr)\bigr  \|_{\mathcal{H}}^2
1_{L_n(s)} \biggr)^{{\trup{p}{2}}}.\nonumber
\end{equation}
This is bounded up to a positive constant by $R_n^{1,1}(t,\bar
{t},x)+R_n^{1,2}(t,\bar{t},x)$, where
%
\begin{eqnarray}
\label{s3.58} R_n^{1,1}(t,\bar{t},x) &=&\E \biggl(\biggl
\llvert \int_t^{\bar{t}} \bigl \| G(\bar{t}-s,x-\ast) A
\bigl(X_n(s,\ast)\bigr) \bigr \|_{\mathcal{H}}^2
1_{L_n(s)}\,\mathrm{d}s \biggr\rrvert \biggr)^{{\trup{p}{2}}}
\nonumber
\\[-8pt]
\\[-8pt]
&=&\E \biggl(\biggl\llvert \int_0^{\bar{t}-t} \bigl \|
G(s,x-\ast) A\bigl(X_n(\bar t-s,\ast)\bigr) \bigr \|_{\mathcal{H}}^2
1_{L_n(s)} \,\mathrm{d}s \biggr\rrvert \biggr)^{{\trup{p}{2}}}
\nonumber
\end{eqnarray}
and
%
\begin{eqnarray}
\label{s3.581} R_n^{1,2}(t,\bar{t},x) &=&\E \biggl(\biggl
\llvert \int_0^t \mathrm{d}s \bigl\llVert
\bigl[G(\bar t-s,x-\ast )
\nonumber
\\[-8pt]
\\[-8pt]
&&\phantom{\E \biggl(\biggl
\llvert \int_0^t \mathrm{d}s \bigl\llVert
\bigl[}{}-G(t-s,x-\ast)\bigr]A\bigl(X_n(s,\ast)\bigr)
\bigr\rrVert ^2_{\mathcal{H}} 1_{L_n(s)}\biggr\rrvert
\biggr)^{{\trup{p}{2}}}.\nonumber
\end{eqnarray}

Using Burkholder and then H\"older inequalities, the linear growth of $A$ and \eqref{s4.18}, we get 

%
\begin{eqnarray}
\label{s3.591}R_n^{1,1}(t,\bar{t},x) &\le& C(t-\bar t)^{\tfrac{p}{2}-1} \int_0^{t-\bar t}ds 
\left( \int_{\R^3} \int_{\R^3} G(s,x-dy)G(s,x-dz)f(y-z)\right)^{p/2}
\nonumber
\\
&&{}\times\Bigl(1+ \sup
_{(t,x)\in[0,T]\times\R^3}\E \bigl( \bigl | X_n(t,x) \bigr |^p1_{L_n(t)}
\bigr) \Bigr)
\\
&\le& C(t-\bar t)^{\tfrac{p}{2}-1} \int_0^{t-\bar t}ds 
\left( \int_{\R^3} \int_{\R^3} G(s,dy)G(s,dz)f(y-z)\right)^{p/2}
\nonumber
\\
&\le& C(t-\bar t)^{\tfrac{p}{2}-1} \int_0^{t-\bar t}  ds 
\left(  \int_{|z|\le 2s} \frac{f(z)}{|z|}dz \right)^{p/2}
\nonumber
\\
\label{s3.60}&\le& C (t-\bar t)^{\tfrac{p}{2}-1} \int_0^{t-\bar t} s^{\nu p/2} ds = C (t-\bar t)^{p\tfrac{\nu+1}{2}}.
\end{eqnarray}

Set $D_n(t,x)=A(X_n(t,x)) 1_{L_n(t)}$. Owing to {\bf Hypothesis 1}, \eqref{s4.18} and Proposition~\ref{pss3.1.1}, 
the conditions \eqref{s3.44}, \eqref{s3.45} of Lemma~\ref{lss3.2.1}
are satisfied with $\kappa\in\, \Big]0,\min\big(\gamma,\gamma_1,\gamma_2,\tfrac{\gamma'}{2}\big) \Big[$. Thus,
%
\begin{equation}
\label{s3.61} R_n^{1,2}(t,\bar{t},x)\le C \bigl( |
\bar{t}-t|^{\rho p}
\bigr),
\end{equation}
with $\rho \in\, \Big]0,\min\Big(\gamma_1,\gamma_2,\gamma,\tfrac{\gamma'}{2},\tfrac{\nu+1}{2},
\tfrac{\rho_1+\kappa}{2},\tfrac{\rho_2}{2} \Big) \Big[$.

It is easy to check that \eqref{s3.60} and
\eqref{s3.61} imply
%
\begin{equation}
\label{s3.63} R_n^1(t,\bar{t},x)\le C|
\bar{t}-t|^{\rho p },
\end{equation}
with $\rho \in\, \Big]0,\min\Big(\gamma_1,\gamma_2,\gamma,\tfrac{\gamma'}{2},\tfrac{\nu+1}{2},
\tfrac{\rho_1+\kappa}{2},\tfrac{\rho_2}{2}\Big) \Big[$.


With the same arguments as those applied in the study of the term
$R_n^2(t,x,\bar x)$ in the proof of Lemma~\ref{lss3.1.2}, we have
\[
R_n^2(t,\bar{t},x) \le C \E \biggl(\int
_0^{\bar t} \mathrm{d}s \bigl \Vert\bigl[G(\bar t-s,x+\ast
)-G(t-s,x-\ast)\bigr] B\bigl(X_n^-(s,\ast)\bigr)\bigr \Vert_{\mathcal{H}}^2
1_{L_n(s)} \biggr)^{\trup{p}{2}}.
\]
This yields $R_n^2(t,\bar{t},x)\le C(R_n^{2,1}(t,\bar
{t},x)+R_n^{2,2}(t,\bar{t},x))$,
where
\begin{eqnarray*}
R_n^{2,1}(t,\bar{t},x)& =& \E \biggl(\int
_0^t \mathrm{d}s \bigl \Vert\bigl[G(\bar t-s,x+
\ast)-G(t-s,x-\ast)\bigr] B\bigl(X_n^-(s,\ast)\bigr)
\bigr \Vert_{\mathcal{H}}^2 1_{L_n(s)} \biggr)^{\trup{p}{2}},
\\
R_n^{2,2}(t,\bar{t},x)& =& \E \biggl(\int
_0^{\bar t-t} \bigl \Vert G(s,x-\ast)B\bigl(X_n^-(s,
\ast)\bigr)\bigr \Vert_{\mathcal{H}}^2 1_{L_n(s)}
\biggr)^{\trup{p}{2}}.
\end{eqnarray*}

The term $R_n^{2,1}(t,\bar{t},x)$ is similar as $R_n^{1,2}(t,\bar
{t},x)$, with $A(X_n)$ replaced by $B(X_n^-)$. Hence both can be
studied using the same approach.
First, we see that the process $D_n(t,x):=B(X_n^-(t,x)) 1_{L_n(t)}$
satisfies the hypothesis of Lemma~\ref{lss3.2.1} with 
$\kappa\in\, \Big]0,\min\big(\gamma,\gamma_1,\gamma_2,\tfrac{\gamma'}{2}\big) \Big[$. 
In fact, this is a consequence of \eqref{s4.18} and Proposition~\ref{pss3.1.2}. Therefore, as for
$R_n^{1,2}(t,\bar{t},x)$, we have
%
\begin{equation}
\label{s3.633} R_n^{2,1}(t,\bar{t},x)\le C |
\bar{t}-t|^{\rho p},
\end{equation}

with $\rho \in\, \Big]0,\min\Big(\gamma_1,\gamma_2,\gamma,\tfrac{\gamma'}{2},\tfrac{\nu+1}{2},
\tfrac{\rho_1+\kappa}{2},\tfrac{\rho_2}{2}\Big) \Big[$.

As for $R_n^{2,2}(t,\bar{t},x)$, it is analogous to $R_n^{1,1}$ with
$A(X_n)$ replaced by $B(X_n^-)$. As in \eqref{s3.60}, we have
%
\begin{equation}
\label{s3.64} R_n^{2,2}(t,\bar{t},x) \le C |
\bar{t}-t|^{p\tfrac{\nu+1}{2}}.
\end{equation}
Consequently, from \eqref{s3.633}, \eqref{s3.64}, we obtain
%
\begin{equation}
\label{s3.640} R_n^2(t,\bar{t},x) \le C |
\bar{t}-t|^{\rho p },
\end{equation}

with $\rho \in\, \Big]0,\min\Big(\gamma_1,\gamma_2,\gamma,\tfrac{\gamma'}{2},\tfrac{\nu+1}{2},
\tfrac{\rho_1+\kappa}{2},\tfrac{\rho_2}{2}\Big) \Big[$.


Let $\hat{B}(X_n(\cdot,\ast))$ be defined by \eqref{s3.351}. Using
Cauchy--Schwarz's inequality and \eqref{s3.101} we have
%
\begin{equation}
\label{s3.641} R_n^3(t,\bar{t},x)\le C
n^{\trup{3p}{2}}2^{n{\trup{p}{2}}} \bigl[R_n^{3,1}(t,
\bar{t},x)+R_n^{3,2}(t,\bar{t},x) \bigr],
\end{equation}
where
\begin{eqnarray*}
R_n^{3,1}(t,\bar{t},x) &=&\E \biggl(\biggl\llvert \int
_0^t \mathrm{d}s \bigl\llVert \bigl[G(
\bar{t}-s,x-\ast )-G(t-s,x-\ast)\bigr] \hat{B}\bigl(X_n(s,\ast)\bigr)
\bigr\rrVert ^2_{\mathcal
{H}}1_{L_n(s)}\biggr\rrvert
\biggr)^{\trup{p}{2}},
\\
R_n^{3,2}(t,\bar{t},x)&=&\E \biggl( \biggl\llvert \int
_0^{\bar{t}-t} \mathrm{d}s\bigl\llVert G(s,x-\ast)
\hat{B}\bigl(X_n(\bar t-s,\ast)\bigr)\bigr\rrVert ^2_{\mathcal{H}}1_{L_n(s)}
\biggr\rrvert \biggr)^{\trup{p}{2}}.
\end{eqnarray*}

From \eqref{s4.19}, it follows that
%
\begin{equation}
\label{s3.65} \sup_{(t,x)\in[0,T\times\IR^3]} \E \bigl( \bigl |\hat {B}
\bigl(X_n(t,x)\bigr)\bigr  |^p1_{L_n(t)} \bigr)\le C
n^{\trup
{3p}{2}} 2^{-np\tfrac{\nu+1}{2}}.
\end{equation}

Let us study $R_n^{3,2}(t,\bar{t},x)$. This term is similar to
$R_n^{1,1}(t,\bar{t},x)$ with $A(X_n)$ replaced here by $\hat
{B}(X_n)$. Hence, as in \eqref{s3.591} we have
%
\begin{eqnarray}
\label{s3.66} R_n^{3,2}(t,\bar{t},x)&\le&  |\bar{t}-t|^{p\tfrac{\nu+1}{2}} \Bigl(\sup
_{(t,x)\in[0,T]\times\IR^3}\E \bigl( \bigl |\hat {B}\bigl(X_n(t,x)\bigr)
\bigr |^p1_{L_n(t)} \bigr) \Bigr)
\nonumber
\\[-8pt]
\\[-8pt]
&\le& C |\bar{t}-t|^{p\tfrac{\nu+1}{2}} n^{\trup
{3p}{2}}2^{-np\tfrac{\nu+1}{2}},
\nonumber
\end{eqnarray}
where in the last inequality we have applied \eqref{s3.65}.

The analysis of $R_n^{3,1}$ relies on a variant of Lemma~\ref
{lss3.2.1} where the process $D_n$ is replaced by $\hat{B}(X_n)$. By
\eqref{s3.65}, this process satisfies a stronger assumption than
\eqref{s3.44}. This fact is expected to compensate the factor
$n^{\trup{3p}{2}}2^{n{\trup{p}{2}}}$ in \eqref{s3.641}.

As in the proof of Lemma~\ref{lss3.2.1}, we consider the decomposition
\[
R_n^{3,1}(t,\bar{t},x)\le\sum
_{k=1}^4 \E \bigl(\bigl |Q^i(t,\bar
{t},x)\bigr |^{{\trup{p}{2}}}1_{L_n(\bar{t})} \bigr),
\]
where $Q^i(t,\bar{t},x)$, $i=1,\ldots,4$, are defined in \eqref{s3.48} and subsequent lines,
with $D_n:=\hat{B}(X_n) 1_{L_n}$.

From \eqref{s3.65} and the triangular inequality, we obtain
%
\begin{equation}
\label{s3.68} \E \biggl(\biggl\llvert \hat{B} \biggl(X_n
\biggl(s,x-\frac{\bar
{t}-s}{t-s}u \biggr) \biggr)-\hat{B}\bigl(X_n(s,x-u)
\bigr) \biggr\rrvert ^p 1_{L_n(s)} \biggr) \le C
n^{\trup{3p}{2}} 2^{-np\tfrac{\nu+1}{2}}.
\end{equation}
Consider the expression \eqref{s3.52-3}  with $D_n=\hat
{B}(X_n)1_{L_n}$. The above estimate \eqref{s3.68} yields
\begin{eqnarray*}
&& \E \bigl( \bigl |Q^1(t,\bar{t},x)\bigr  |^{{\trup{p}{2}}}1_{L_n(\bar
{t})}\bigr)
\\
&&\quad \le C n^{\trup{3p}{2}}2^{-np\tfrac{\nu+1}{2}} \left(\int_0^t
\mathrm{d}s\int_{S^2} \int_{S^2} \bigl(s+h\bigr)^2 
f\bigl([s+h]\xi-[s+h]\eta\bigl) \sigma(d\xi)\sigma(d\eta) \right)^{\tfrac{p}{2}}.
\end{eqnarray*}
This implies
%
\begin{equation}
\label{s3.69} \E \bigl( \bigl |Q^1(t,\bar{t},x)\bigr  |^{{\trup{p}{2}}}1_{L_n(\bar
{t})}
\bigr)\le Cn^{\trup{3p}{2}}2^{-np\tfrac{\nu+1}{2}}.
\end{equation}


Consider the procedure to get the expressions \eqref{s3.52-1} and \eqref{s3.52-2}, with $D_n=\hat{B}(X_n)
1_{L_n}$. Using \eqref{s3.351}, \eqref{s3.65}, \eqref{s3.68} and
\eqref{s4.19}, we obtain
%
\begin{equation}
\label{s3.70} \E \bigl( \bigl |Q^2(t,\bar{t},x)\bigr  |^{{\trup{p}{2}}}1_{L_n(\bar
{t})}
\bigr) \le Cn^{\trup{3p}{2}}2^{-np\tfrac{\nu+1}{2}}.
\end{equation}

Similarly,
%
\begin{equation}
\label{s3.71} \E \bigl( \bigl |Q^3(t,\bar{t},x)\bigr  |^{{\trup{p}{2}}}1_{L_n(\bar
{t})}
\bigr) \le Cn^{\trup{3p}{2}}2^{-np\tfrac{\nu+1}{2}}.
\end{equation}

Let us now consider the procedure to estimate $Q^4$ in Lemma~\ref{lss3.2.1}, with $D_n=\hat
{B}(X_n) 1_{L_n}$. Appealing to \eqref{s3.65}, we obtain
%
\begin{equation}
\label{s3.72} \E \bigl( \bigl |Q^4(t,\bar{t},x)\bigr  |^{{\trup{p}{2}}}1_{L_n(\bar
{t})}
\bigr) \le C n^{\trup{3p}{2}} 2^{-np\tfrac{\nu+1}{2}}.
\end{equation}


From \eqref{s3.69}--\eqref{s3.72} it follows that
%
\begin{equation}
\label{s3.73} R_n^{3,1}(t,\bar t,x) \le
Cn^{\trup{3p}{2}}2^{-np\tfrac{\nu+1}{2}},
\end{equation}
where $C$ is a finite constant.

Set $f_n:=n^{3p}2^{-np\tfrac{\nu+1}{2}}2^{\trup{1}{2}}=n^{3p}2^{-np\tfrac{\nu}{2}} $.
From \eqref{s3.641}, \eqref{s3.66}, \eqref{s3.73}, it follows that\vspace*{-1pt}
%
\begin{equation}
\label{s3.74} R_n^3(t,\bar{t},x)\le C |\bar{t}-t
|^{\rho p}+Cf_n,\qquad \rho \in\, \biggl]0,\frac{\nu+1}{2} \biggr[.
\end{equation}
%

By applying Cauchy--Schwarz's inequality, we see that\vspace*{-1pt}
\[
R_n^4(t,x,\bar{x}) \le C\E \biggl(\int
_0^{\bar t} \mathrm{d}s \bigl\llVert \bigl[G(\bar
t-s,x-\ast )-G(t-s,x-\ast)\bigr] D\bigl(X_n(s,\ast)\bigr)\bigr\rrVert
_{\mathcal
{H}}^21_{L_n(s)} \biggr)^{\trup{p}{2}}.
\]
The last expression is similar as \eqref{s3.571} with the function $A$
replaced by $D$. Therefore, as in \eqref{s3.63} we obtain\vspace*{-1pt}
%
\begin{equation}
\label{s3.75} R_n^4(t,\bar{t},x)\le C |
\bar{t}-t|^{\rho p },
\end{equation}
with $\rho \in\, \Big]0,\min\Big(\gamma_1,\gamma_2,\gamma,\tfrac{\gamma'}{2},\tfrac{\nu+1}{2},
\tfrac{\rho_1+\kappa}{2},\tfrac{\rho_2}{2}\Big) \Big[$.


Finally, we consider $R_n^5(t,\bar{t},x)$. Clearly,\vspace*{-1pt}
\[
R_n^5(t,\bar{t},x)\le C \bigl[R_n^{5,1}(t,
\bar{t},x)+R_n^{5,2}(t,\bar {t},x) \bigr],
\]
where\vspace*{-1pt}
\begin{eqnarray*}
R_n^{5,1}(t,\bar{t},x) &:=&\E \biggl(\biggl\llvert \int
_0^t \int_{\R^3} \bigl[G(
\bar{t}-s,x-\mathrm{d}y)-G(t-s,x-\mathrm{d}y)\bigr] b\bigl(X_n(s,y)
\bigr) \,\mathrm{d}s \biggr\rrvert ^p 1_{L_n(\bar{t})} \biggr),
\\
R_n^{5,2}(t,\bar{t},x)&:=&\E \biggl(\biggl\llvert \int
_t^{\bar{t}} \int_{\R
^3} G(
\bar{t}-s,x-\mathrm{d}y) b\bigl(X_n(s,y)\bigr) \,\mathrm{d}s \biggr
\rrvert ^p1_{L_n(\bar{t})} \biggr).
\end{eqnarray*}

Applying the change of variable, $y\mapsto\frac{y-x}{\overline
{t}-s}+ x $ and $y\mapsto\frac{y-x}{t-s}+x$, we see that\vspace*{-1pt}
\[
R_n^{5,1}(t,\bar{t},x)= \E \bigl( \bigl |T_1(t,
\bar{t},x)-T_2(t,\bar {t},x) \bigr |^p 1_{L_n(\bar{t})} \bigr),
\]
where\vspace*{-1pt}
\begin{eqnarray*}
T_1(t,\bar{t},x)&=&\int_0^t (
\bar{t}-s)\int_{\R^3} G(1,x-\mathrm{d}y) b\bigl(X_n
\bigl(s,(\bar{t}-s) (y-x)+x \bigr)\bigr) \,\mathrm{d}s,
\\
T_2(t,\bar{t},x)&=&\int_0^t (t-s)
\int_{\R^3} G(1,x-\mathrm{d}y) b\bigl(X_n
\bigl(s,(t-s) (y-x)+x\bigr)\bigr)\,\mathrm{d}s.
\end{eqnarray*}
By adding and subtracting $t$ in $T_1$ we get\vspace*{-1pt}
\begin{eqnarray*}
T_1(t,\bar{t},x)&=&\int_0^t (
\bar{t}-t) \int_{\R^3} G(1,x-\mathrm{d}y) b
\bigl(X_n\bigl(s,(\bar{t}-s) (y-x)+x \bigr)\bigr) \,\mathrm{d}s
\\
&&{}+\int_0^t (t-s)\int_{\R^3}
G(1,x-\mathrm{d}y) b\bigl(X_n\bigl(s,(\bar{t}-s) (y-x)+x \bigr)\bigr)
\,\mathrm{d}s.
\end{eqnarray*}
Then, H\"older's inequality yields
\begin{eqnarray*}
&&R_n^{5,1}(t,\bar{t},x)
\\
&&\quad\le  C |\bar{t}-t|^p
\int_0^t \mathrm{d}s\int_{\R^3}
G(1,x-\mathrm{d}y)\E \bigl( \bigl |b \bigl(X_n\bigl(s,(\bar{t}-s) (y-x)+x
\bigr) \bigr) \bigr |^p1_{L_n(s)} \bigr)
\\
 &&\qquad {} + C \int_0^t |t-s|^{p}\,
\mathrm{d}s\int_{\R^3} G(1,x-\mathrm{d}y)\E \bigl( \bigl |b
\bigl(X_n\bigl(s,(\bar{t}-s) (y-x)+x \bigr) \bigr)
\\
&&\phantom{\qquad {} + C \int_0^t |t-s|^{p}\,
\mathrm{d}s\int_{\R^3} G(1,x-\mathrm{d}y)\E \bigl( \bigl |} {} -b
\bigl(X_n\bigl(s,(t-s) (y-x)+x \bigr) \bigr) \bigr |^p
1_{L_n(s)} \bigr).
\end{eqnarray*}

Owing to \eqref{s4.18}, the first term on the right hand-side of the
last inequality is bounded up to a constant by $|\bar{t}-t|^p$. For
the second one, we use the Hypothesis~\ref{HypB} along with
\eqref{s3.17} to obtain
\begin{eqnarray*}
&&\int_0^t |t-s|^{p} \,
\mathrm{d}s\int_{\R^3} G(1,x-\mathrm{d}y)
\\
&&\qquad{} \times\E \bigl( \bigl |b \bigl(X_n\bigl(s,(\bar{t}-s) (y-x)+x
\bigr) \bigr)-b \bigl(X_n\bigl(s,(t-s) (y-x)+x \bigr) \bigr)
\bigr |^p 1_{L_n(s)} \bigr)
\\
&&\quad \le C \int_0^t \mathrm{d}s \int
_{\R^3} G(1,x-\mathrm{d}y)
\\
&&\qquad{} \times\E \bigl( \bigl |X_n\bigl(s,(\bar{t}-s) (y-x)+x
\bigr)-X_n\bigl(s,(t-s) (y-x)+x \bigr)\bigr  |^p
1_{L_n(s)} \bigr)
\\
&&\quad \le C |t-\bar t|^{\rho p},
\end{eqnarray*}
with $\rho\in\, \Big]0,\min\big(\gamma,\gamma_1,\gamma_2,\tfrac{\gamma'}{2}\big) \Big[$.

H\"older inequality along with \eqref{s4.18} clearly yields
%
\begin{eqnarray}
\label{s3.56} R_n^{5,2}(t,\bar{t},x)&\le& C|
\bar{t}-t|^{p-1} \int_t^{\bar{t}} \int
_{\R^3} G(\bar{t}-s,x-\mathrm{d}y) \E \bigl(\bigl\llvert b
\bigl(X_n(s,y)\bigr) \bigr\rrvert ^p1_{L_n(s)}
\bigr)\,\mathrm{d}s
\nonumber
\\[-8pt]
\\[-8pt]
&\le& C|\bar{t}-t|^{p}.
\nonumber
\end{eqnarray}

Hence, we have proved that
%
\begin{equation}
\label{s3.57} R_n^5(t,\bar{t},x)\le C |\bar{t}-t
|^{\rho p},
\end{equation}
where $\rho\in\, \Big]0,\min\big(\gamma,\gamma_1,\gamma_2,\tfrac{\gamma'}{2}\big) \Big[$.

With the inequalities \eqref{s3.63}, \eqref{s3.640}, \eqref{s3.74},
\eqref{s3.75} and \eqref{s3.57}, we have
\[
\E \bigl(\bigl |X_n(\bar{t},x)-X_n(t,x)\bigr |^p1_{L_n(\bar{t})}
\bigr)\le C \Bigl[|\bar{t}-t|^{\rho p } + f_n \Bigr],
\]
with $\rho \in\, \Big]0,\min\Big(\gamma_1,\gamma_2,\gamma,\tfrac{\gamma'}{2},\tfrac{\nu+1}{2},
\tfrac{\rho_1+\kappa}{2},\tfrac{\rho_2}{2}\Big) \Big[$.

For a given fixed $\bar t\in[t_0,T]$, we introduce the function
\[
\varPsi_{n,x,p}^{\bar{t}}(t):=\E \bigl(\bigl |X_n(\bar
{t},x)-X_n(t,x)\bigr |^p1_{L_n(\bar{t})} \bigr),
\]
for $t_0\le t\le\bar{t}$.\vadjust{\goodbreak}

Notice that $\lim_{n\to\infty} f_n=0$ and thus, $\sup_n f_n \le C$.
Thus, there exists
a constant $0<C_0<\infty$, such that
\[
\sup_n f_n\le C_0 t_0
\le C_0\bar{t}\le C_0 \int_0^{\bar{t}}
\mathrm{d}s \bigl[1+\varPsi_{n,x,p}^{\bar{t}}(s) \bigr].
\]
With a similar argument, there exists $0<C_1<\infty$ such that
\[
1\le C_1 t_0 \le C_1\bar{t}\le
C_1 \int_0^{\bar{t}} \mathrm{d}s
\bigl[1+\varPsi_{n,x,p}^{\bar{t}}(s) \bigr].
\]
Therefore,
\[
1+ \varPsi_{n,x,p}^{\bar{t}}(t) \le C \biggl\{|
\bar{t}-t|^{\rho p} + \int_0^{\bar{t}}
\mathrm{d}s \bigl[1+\varPsi_{n,x,p}^{\bar{t}}(s) \bigr] \biggr\}.
\]
Then, by Gronwall's lemma,
\[
1+ \varPsi_{n,x,p}^{\bar{t}}(t) \le C \bigl(|\bar{t}-t|^{\rho p}
\bigr),
\]
where $\rho \in\, \Big]0,\min\Big(\gamma_1,\gamma_2,\gamma,\tfrac{\gamma'}{2},\tfrac{\nu+1}{2},
\tfrac{\rho_1+\kappa}{2},\tfrac{\rho_2}{2}\Big) \Big[$. This finish the
proof of the proposition.
\end{proof}


\subsection{Pointwise convergence}
\label{ss3.3}
For the sake of completeness we will sketch the proof of the Theorem~\ref
{ts3.3}. The proofs is very similar to the Theorem 2.4 in \cite{Delgado--Sanz-Sole012} just with small changes.
The main difference is the bound used, in \cite{Delgado--Sanz-Sole012} they used the bound $2^{-np(3-\beta)/2}$ while here
we use $2^{-np\tfrac{\nu+1}{2}}$ (see \eqref{s4.4} and \eqref{s4.19}). 
 
Using equations \eqref{s3.7}, \eqref{s3.6}, we write the
difference $X_n(t,x)-X(t,x)$ grouped into comparable terms
in order to prove their convergence to zero.

Notice that the initial conditions vanish so we could do the following decomposition.

\[
X_n(t,x)-X(t,x)=\sum_{i=1}^8
U_n^i(t,x),
\]
where
\begin{eqnarray*}
U_n^1(t,x) &=& \int_0^t
\int_{\R^3} G(t-s,x-y) \bigl[(A+B) \bigl(X_n(s,y)
\bigr)-(A+B) \bigl(X(s,y)\bigr) \bigr] M(\mathrm{d}s,\mathrm{d}y),
\\
U_n^2(t,x) &=& \bigl\langle G(t-\cdot,x-\ast)\bigl[D
\bigl(X_n(\cdot,\ast)\bigr)-D\bigl(X(\cdot ,\ast)\bigr)\bigr], h \bigr
\rangle_{\mathcal{H}_t},
\\
U_n^3(t,x) &=& \int_0^t
\mathrm{d}s\int_{\R^3} G(t-s,x-\mathrm{d}y)\bigl[b
\bigl(X_n(s,y)\bigr)-b\bigl(X(s,y)\bigr)\bigr],
\\
U_n^4(t,x) &=& \bigl\langle G(t-\cdot,x-\ast)\bigl[B
\bigl(X_n(\cdot,\ast )\bigr)-B\bigl(X_n^-(\cdot,\ast)
\bigr)\bigr], w^n \bigr\rangle_{\mathcal{H}_t},
\\
U_n^5(t,x) &=& \bigl\langle G(t-\cdot,x-\ast)\bigl[B
\bigl(X_n^-(\cdot,\ast )\bigr)-B\bigl(X^-(\cdot,\ast)\bigr)\bigr],
w^n \bigr\rangle_{\mathcal{H}_t},
\\
U_n^6(t,x) &=& \bigl\langle G(t-\cdot,x-\ast)B
\bigl(X^-(\cdot,\ast)\bigr), w^n \bigr\rangle_{\mathcal{H}_t}
\\
&&{} - \int_0^t\int_{\R^3}
G(t-s,x-y)B\bigl(X^-(s,y)\bigr) M(\mathrm{d}s,\mathrm{d}y),
\\
U_n^7(t,x) &=& \int_0^t
\int_{\R^3} G(t-s,x-y)\bigl[B\bigl(X^-(s,y)\bigr)-B
\bigl(X_n^-(s,y)\bigr)\bigr] M(\mathrm{d}s,\mathrm{d}y),
\\
U_n^8(t,x) &=& \int_0^t
\int_{\R^3} G(t-s,x-y)\bigl[B\bigl(X_n^-(s,y)
\bigr)-B\bigl(X_n(s,y)\bigr)\bigr] M(\mathrm{d}s,\mathrm{d}y).
\end{eqnarray*}
Here, we have used the abridged notation $X^-(\cdot,\ast)$ for the
stochastic process $X^-(t,x):=X(t,t_n,x)$ defined in \eqref{s3.8.3}.
Notice that, although this is not apparent in the notation $X^-(\cdot
,\ast)$ does depend on $n$.

Fix $p\in[2,\infty[$ we get
\[
\E \bigl(\bigl\llvert X_n(t,x)-X(t,x)\bigr\rrvert
^p1_{L_n(t)} \bigr)\le C\sum_{i=1}^8
\E \bigl(\bigl\llvert U_n^i(t,x)\bigr\rrvert
^p1_{L_n(t)} \bigr).
\]
Next, we analyze the contribution of each term $U_n^i(t,x)$,
$i=1,\ldots,8$. By following the procedure in \cite{Delgado--Sanz-Sole012} we have,

%
\begin{eqnarray}
\label{s3.76} \E \bigl(\bigl\llvert U_n^1(t,x)\bigr
\rrvert ^p1_{L_n(t)} \bigr)&&+\E \bigl(\bigl\llvert U_n^2(t,x)\bigr
\rrvert ^p1_{L_n(t)} \bigr)+\E \bigl(\bigl\llvert U_n^3(t,x)\bigr
\rrvert ^p1_{L_n(t)} \bigr)
\nonumber
\\
&&\le C \int_0^t
\mathrm{d}s \Bigl[\sup_{y \in K(s)} \E \bigl( \bigl |X_n(s,y)-X(s,y)
\bigr |^{p} 1_{L_n(s)} \bigr) \Bigr].
\end{eqnarray}
For $U_n^5(t,x),U_n^7(t,x)$ we follow the same procedure as in \cite{Delgado--Sanz-Sole012} and we arrive to 
\begin{eqnarray*}
\E \bigl[ \bigl(\bigl\llvert U_n^5(t,x) +
U_n^7(t,x)\bigr\rrvert ^p
\bigr)1_{L_n(t)} \bigr] &\le& C \biggl\{\int_0^t
\mathrm{d}s \Bigl[ \sup_{y\in K(s)}\E \bigl( \bigl |X_n^-(s,y)-X_n(s,y)
\bigr |^p 1_{L_n(s)} \bigr) \Bigr]
\\
&&\phantom{C \biggl\{} {}   +\int_0^t
\mathrm{d}s \Bigl[ \sup_{y\in K(s)}\E \bigl( \bigl |X_n(s,y)-X(s,y)
\bigr |^p 1_{L_n(s)} \bigr) \Bigr]
\\
&&\phantom{C \biggl\{} {}   + \int_0^t
\mathrm{d}s \Bigl[ \sup_{y\in K(s)}\E \bigl( \bigl |X(s,y)-X^-(s,y)
\bigr |^p 1_{L_n(s)} \bigr) \Bigr] \biggr\}.
\end{eqnarray*}
Recall that $X^-(s,y)=X(s,s_n,y)$. By applying \eqref{s4.4} and \eqref
{s4.19}, we obtain
%
\begin{eqnarray}
\label{s3.80} \E \bigl[ \bigl(\bigl\llvert U_n^5(t,x)
+ U_n^7(t,x)\bigr\rrvert ^p
\bigr)1_{L_n(t)} \bigr] &\le& C \int_0^T
\mathrm{d}s \Bigl[ \sup_{y\in K(s)}\E \bigl( \bigl |X_n(s,y)-X(s,y)
\bigr |^p 1_{L_n(s)} \bigr) \Bigr]\qquad
\nonumber
\\[-8pt]
\\[-8pt]
&&{} + Cn^{\trup{3p}{2}} 2^{-np\tfrac{\nu+1}{2}}.
\nonumber
\end{eqnarray}

Next, we will see that
%
\begin{equation}
\label{s3.81} \lim_{n\to\infty} \Bigl(\sup_{t\in[0,T]}
\sup_{x\in K(t)}\E \bigl(\bigl\llvert U_n^i(t,x)
\bigr\rrvert ^p1_{L_n(t)} \bigr) \Bigr)=0,\qquad i=4,6,8.
\end{equation}

Consider $i=4$. Cauchy--Schwarz' inequality along with \eqref{s3.101} implies
\begin{eqnarray*}
&&\E \bigl(\bigl\llvert U_n^4(t,x)\bigr\rrvert
^p1_{L_n(t)} \bigr)
\\
&&\quad\le Cn^{\trup{3p}{2}}2^{n{\trup{p}{2}}} \E
\biggl(\int_0^t \mathrm{d}s\bigl \| G(t-s,x-\ast)
\bigl[B(X_n)-B\bigl(X_n^-\bigr)\bigr](s,
\ast)1_{L_n(s)} \bigr \|_{\mathcal{H}}^2 \biggr)^{{\trup{p}{2}}}.
\end{eqnarray*}
Then, the Lipschitz continuity of $B$ and \eqref{s4.19} yield
\begin{eqnarray*}
\E \bigl(\bigl\llvert U_n^4(t,x)\bigr\rrvert
^p1_{L_n(t)} \bigr) &\le& Cn^{\trup{3p}{2}}2^{n{\trup{p}{2}}}
\int_0^t \mathrm{d}s \Bigl[ \sup
_{y\in\IR^3}\E \bigl( \bigl |X_n(s,y)-X_n^-(s,y)
\bigr |^p 1_{L_n(s)} \bigr) \Bigr]
\\
&\le& C n^{3p} 2^{-np\tfrac{\nu+1-1}{2}}.\\
&\le& C n^{3p} 2^{-np\tfrac{\nu}{2}}.
\end{eqnarray*}
Since $\nu\in\,]0,1]$, this implies \eqref{s3.81} for $i=4$.

The arguments based on Burkholder's and H\"older's inequalities,
already applied many times, give\vspace*{-2pt}
\begin{eqnarray*}
\E \bigl(\bigl\llvert U_n^8(t,x)\bigr\rrvert
^p1_{L_n(t)} \bigr) &\le& C\int_0^t
\mathrm{d}s \sup_{y\in\IR^3} \E \bigl( \bigl |X_n^-(s,y)-X_n(s,y)
\bigr |^{p} 1_{L_n(s)} \bigr)
\\
&\le& C n^{\trup{3p}{2}} 2^{-np\tfrac{\nu+1}{2}},
\end{eqnarray*}
where, in the last inequality we have used \eqref{s4.19}. Thus, \eqref
{s3.81} holds for $i=8$.


Let us now consider the case $i=6$. Define\vspace*{-2pt}
\begin{eqnarray*}
U_n^{6,1}(t,x) &=&\int_0^t
\int_{\R^3} \bigl\{\pi_n \bigl[
\tau_n \bigl[G(t-\cdot ,x-\ast)B\bigl(X^-(\cdot,\ast)\bigr) \bigr]
\\
&&\phantom{\int_0^t
\int_{\R^3} \bigl\{\pi_n \bigl[}{}  -G(t-\cdot,x-\ast)\tau_n \bigl[B\bigl(X^-(\cdot,
\ast )\bigr) \bigr] \bigr](s,y) \bigr\} M(\mathrm{d}s,\mathrm{d}y),
\\
U_n^{6,2}(t,x) &=& \int_0^t
\int_{\R^3} \pi_n \bigl[G(t-\cdot,x-\ast)
\tau_n \bigl[B\bigl(X^-(\cdot,\ast)\bigr) \bigr]
\\[-1pt]
&&\phantom{\int_0^t
\int_{\R^3} \pi_n \bigl[}{}   - G(t-\cdot,x-\ast)B\bigl(X^-(\cdot,\ast)\bigr) \bigr](s,y)M(
\mathrm{d}s,\mathrm{d}y),
\\[-1pt]
U_n^{6,3}(t,x) & =& \int_0^t
\int_{\R^3} \bigl\{ \pi_n \bigl[G(t-\cdot,x-\ast
)B\bigl(X^-(\cdot,\ast)\bigr) \bigr]
\\[-1pt]
&&\phantom{\int_0^t
\int_{\R^3} \bigl\{}{}   -G(t-s,x-y)B\bigl(X^-(s,y)\bigr) \bigr\} M(\mathrm{d}s,
\mathrm{d}y).
\end{eqnarray*}
Clearly,\vspace*{-2pt}
\[
U_n^6(t,x)=U_n^{6,1}(t,x)+U_n^{6,2}(t,x)+U_n^{6,3}(t,x).
\]

In a rather similar way to \cite{Delgado--Sanz-Sole012} we establish
\begin{eqnarray*}
\E\bigl(|U_n^{6,1}(t,x)|^p 1_{L_n(t)}\bigr) &\le &
 C g_n.
\end{eqnarray*}
where $g_n$ is a bounded sequence that converges to zero.
Thus, we have established the convergence
%
\begin{equation}
\label{s3.85} \lim_{n\to\infty} \sup_{(t,x)\in[0,T]\times K(t)} \E
\bigl(\bigl\llvert U_n^{6,1}(t,x)\bigr\rrvert
^p1_{L_n(t)} \bigr)=0.
\end{equation}
Next, we consider the term $U_n^{6,2}(t,x)$. As usually for these type
of terms, we apply Burkholder's and then H\"older's inequalities, along
with the contraction property of the projection\vadjust{\goodbreak} $\pi_n$. This yields,
\begin{eqnarray*}
&& \E \bigl(\bigl\llvert U_n^{6,2}(t,x)\bigr\rrvert
^p 1_{L_n(t)} \bigr)
\\
&&\quad =\E \biggl(\biggl\llvert \int_0^t\int
_{\R^3}\pi_n \bigl[G(t-\cdot ,x-\ast) \bigl\{B
\bigl(X^-\bigl( \bigl(\cdot+2^{-n}\bigr)\wedge t,\ast\bigr)\bigr) - B
\bigl(X^-(\cdot,\ast )\bigr) \bigr\} \bigr](s,y)
\\
&&\phantom{\quad =\E \biggl(\biggl\llvert}{}   \times M(\mathrm{d}s,\mathrm{d}y)\biggr\rrvert
^p1_{L_n(t)} \biggr)
\\
&&\quad \le C \int_0^t \sup
_{x\in\IR^3}\E \bigl(\bigl\llvert X\bigl(\bigl(s+2^{-n}
\bigr)\wedge t,\bigl(s_n+2^{-n}\bigr)\wedge t,x
\bigr)-X(s,s_n,x)\bigr\rrvert \bigr)^p.
\end{eqnarray*}
Equation \eqref{s3.7} is a particular case of equation \eqref{s3.6}.
Therefore, Proposition~\ref{pss3.2.1} also holds with $X_n$ replaced
by $X$. Then,
by virtue of \eqref{s4.4} and \eqref{s3.43}, this is bounded up to a
constant by $2^{-np \tfrac{1+\nu}{2}}+2^{-np\rho}$, for some $\rho>0$
Consequently,
%
\begin{equation}
\label{s3.86} \lim_{n\to\infty} \sup_{(t,x)\in[0,T]\times\R^3} \E
\bigl(\bigl\llvert U_n^{6,2}(t,x)\bigr\rrvert
^p1_{L_n(t)} \bigr)=0.
\end{equation}


For $U_n^{6,3}(t,x)$, after having applied Burkholder's inequatily we have
\[
\E \bigl( \bigl | U_n^{6,3}(t,x) \bigr |^p
1_{L_n(t)} \bigr)\le C\E \bigl(\bigl\llVert (\pi_n-I_{\mathcal{H}_t}
) \bigl[G(t-\cdot,x-\ast)B\bigl(X^-(\cdot,\ast)\bigr) \bigr] 1_{L_n(\cdot)}
\bigr\rrVert _{\mathcal{H}_t}^p \bigr).
\]
We want to prove that the right-hand side of this inequality tends to
zero as $n\to\infty$, uniformly in $(t,x)\in[t_0,T]\times K(t)$. 
By using a similar approach as in \cite{milletss2}, pages 906--909.

Set
\[
\tilde Z_n(t,x)= \bigl\llVert (\pi_n-I_{\mathcal{H}_t}
) \bigl[G(t-\cdot,x-\ast)B\bigl(X^-(\cdot,\ast)\bigr) \bigr] 1_{L_n(\cdot)}
\bigr\rrVert _{\mathcal{H}_t}.
\]
Since $\pi_n$ is a projection on the Hilbert space $\mathcal{H}_t$,
the sequence $\{\tilde Z_n(t,x), n\ge1\}$ decreases to zero as $n\to
\infty$.
Assume that
%
\begin{equation}
\label{s3.87} \E \Bigl(\sup_n\bigl\llVert G(t-\cdot,x-
\ast) B\bigl(X^-(\cdot,\ast)\bigr) 1_{L_n(\cdot)} \bigr\rrVert
_{\mathcal{H}_t}^p \Bigr) <\infty.
\end{equation}
Remember that $X^-(s,y)$ stands for $X(s,s_n,y)$, defined in \eqref
{s3.8.3}, and therefore it depends on~$n$.
Then, by bounded convergence, this would imply $\lim_{n\to\infty}\E
(\tilde Z_n(t,x) )^p=0$.
Set $Z_n(t,x)=\E (\tilde Z_n(t,x) )^p$. Proceeding as in
the proof of Lemmas~\ref{lss3.1.1}, \ref{lss3.2.1}, we can check that\vspace*{-1pt}
\[
\bigl\llvert \bigl(Z_n(t,x)\bigr)^{\trup{1}{p}}-
\bigl(Z_n(\bar t,\bar x)\bigr)^{\trup
{1}{p}}\bigr\rrvert \le C \bigl(
\vert t-\bar t\vert+\vert x-\bar x\vert \bigr)^\rho,
\]
for some $\rho>0$.

Hence, $(Z_n)_n$ is a sequence of monotonically decreasing continuous
functions defined on
$[0,T]\times\IR^3$ which converges pointwise to zero. Appealing to
Dini's theorem, we obtain
%
\begin{equation}
\label{s3.88} \lim_{n\to\infty}\sup_{(t,x)\in[t_0,T]\times K(t)}\E
\bigl(\tilde Z_n(t,x) \bigr)^p=0.
\end{equation}
This yields the expected result on $U_n^{6,3}$.\vadjust{\goodbreak}

To prove \eqref{s3.87} just follow the procedure in \cite[page 2205]{Delgado--Sanz-Sole012}.

In order to conclude the proof, let us consider the estimates 
\eqref{s3.76}, \eqref{s3.80}, along with
\eqref{s3.81}. We see that
\[
\E \bigl(\bigl\llvert X_n(t,x)-X(t,x)\bigr\rrvert
^p1_{L_n(t)} \bigr)\le C_1 \theta_n +
C_2 \int_0^t \mathrm{d}s \Bigl[
\sup_{x\in K(s) }\E \bigl(\bigl\llvert X_n(s,x)-X(s,x)
\bigr\rrvert ^p 1_{L_n(s)} \bigr) \Bigr],
\]
where $(\theta_n,n\ge1)$ is a sequence of real numbers which
converges to zero as $n\to\infty$. Applying Gronwall's lemma, we
finish the proof of the theorem.
\qed


\subsection{Proof of Theorem \texorpdfstring{\protect\ref{ts3.1}}{2.2}}
\label{ss3.4}

Fix $t_0>0$ and a compact set $K\subset\IR^3$.
Let $Y_n(t,x):= X_n(t,x)-X(t,x)$ and $B_n(t):= L_n(t)$, $n\ge1$,
$(t,x)\in[t_0,T]\times K$, $p\in[1,\infty[$. From Theorems~\ref
{ts3.2} and~\ref{ts3.3}, we see that the conditions (P1) and (P2) of
Lemma~\ref{ls5.2} are satisfied with $\delta=p\rho-4$, for any $\rho \in\, 
\Big]0,\min\Big(\gamma_1,\gamma_2,\gamma,\tfrac{\gamma'}{2},\tfrac{\nu+1}{2},
\tfrac{\rho_1+\kappa}{2},\tfrac{\rho_2}{2}\Big) \Big[$, where 
$\kappa \in\, \Big]0,\min\Big(\gamma_1,\gamma_2,\gamma,\tfrac{\gamma'}{2}\Big) \Big[$.
We infer that
%
\begin{equation}
\label{s5.2} \lim_{n\to\infty}\E \bigl(\llVert X_n-X
\rrVert ^p_{\rho
,t_0,K}1_{L_n(t)} \bigr)=0,
\end{equation}
for any $p\in[1,\infty[$ and $\rho \in\, \Big]0,\min\Big(\gamma_1,\gamma_2,\gamma,\tfrac{\gamma'}{2},\tfrac{\nu+1}{2},
\tfrac{\rho_1+\kappa}{2},\tfrac{\rho_2}{2}\Big) \Big[$.

Fix $\epsilon>0$. Since $\lim_{n\to\infty}\Pb(L_n(t)^c)=0$, there
exists $N_0\in\IN$ such that for all $n\ge N_0$,
$\Pb(L_n(t)^c)<\epsilon$.  Then,
for any $\lambda>0$ and $n\ge N_0$,
\begin{eqnarray*}
\Pb \bigl(\|X_n-X\|_{\rho,t_0,K}>\lambda \bigr)&\le&\epsilon+ \Pb \bigl( \bigl(
\|X_n-X\|_{\rho,t_0K}>\lambda \bigr)\cap L_n(t) \bigr)
\\
&\le& \epsilon+ \lambda^{-p} \E \bigl(\|X_n-X
\|_{\rho
,t_0K}^p1_{L_n(t)} \bigr).
\end{eqnarray*}
Since $\epsilon>0$ is arbitrary, this finishes the proof of the theorem.
\qed


\section{Support theorem}
\label{sm}

This section is devoted to the characterization of the topological
support of the law of the random field solution to the stochastic wave
equation \eqref{s1.8}. As has been explained in the \hyperref[s1]{Introduction}, this
is a corollary of Theorem~\ref{ts3.1}.
%
\begin{theorem}
\label{tsm.1}
Assume that the functions $\varsigma$ and $b$ are Lipschitz continuous.
Fix $t_0\in\,]0,T[$ and a compact set $K\subset\IR^3$. Let $u=\{
u(t,x), (t,x)\in[t_0,T]\times K\}$ be the random field solution to
\eqref{s1.8}. Fix
$\rho \in\, \Big]0,\min\Big(\gamma_1,\gamma_2,\gamma,\tfrac{\gamma'}{2},\tfrac{\nu+1}{2},
\tfrac{\rho_1+\kappa}{2},\tfrac{\rho_2}{2}\Big) \Big[$, where 
$\kappa \in\, \Big]0,\min\Big(\gamma_1,\gamma_2,\gamma,\tfrac{\gamma'}{2}\Big) \Big[$. Then the topological
support of the law of $u$ in the space $\mathcal{C}^\rho
([t_0,T]\times K)$ is the closure in
$\mathcal{C}^\rho([t_0,T]\times K)$ of the set of functions $\{\Phi
^h, h\in\mathcal{H}_T\}$, where $\{\Phi^h(t,x), (t,x)\in
[t_0,T]\times K\}$ is the solution of \eqref{sm.h}.
\end{theorem}

Let $\{w^n, n\ge1\}$ be the sequence of $\mathcal{H}_T$-valued random
variables defined in \eqref{s3.3}. For any $h\in\mathcal{H}_T$, we
consider the sequence of transformations of $\Omega$ defined in \eqref
{sm.1}. As has been pointed out in Section~\ref{s1},
$P\circ(T_n^h)^{-1}\ll P$.

Notice also that the process $v_n(t,x):=(u\circ T_n^h)(t,x)$, $(t,x)\in
[t_0,T]\times\IR^3$, satisfies the equation
%
\begin{eqnarray}
\label{sm.2} v_n(t,x)&=&\int_0^t
\int_{\IR^3}G(t-s,x-y)\varsigma \bigl(v_n(s,y)\bigr)M(
\mathrm{d}s,\mathrm{d}y)
\nonumber
\\[-8pt]
\\[-8pt]
&&{}+ \bigl\langle G(t-\cdot,x-\ast)\varsigma\bigl(v_n(\cdot,\ast )
\bigr),h-w^n \bigr\rangle_{\mathcal{H}_t} +\int_0^t
\mathrm{d}s \bigl[G(t-s,\cdot)\star b\bigl(v_n(s,\cdot)\bigr)
\bigr](x).\quad\ \
\nonumber
\end{eqnarray}

\begin{proof}{of Theorem~\ref{tsm.1}}
According to the method developed in \cite{millet-ss94a} (see also
\cite{Ba-Mi-SS} and Section 1 in \cite{Delgado--Sanz-Sole012} for a summary), the theorem will
be a consequence of the following convergences:
%
\begin{eqnarray}
\label{sm.3}\lim_{n\to\infty}\Pb \bigl\{\bigl\llVert u-
\Phi^{w^n}\bigr\rrVert _{\rho,t_0,K}>\eta \bigr\}&=&0,
\\
\label{sm.4}\lim_{n\to\infty}\Pb \bigl\{\bigl\llVert u\circ
T_n^h-\Phi^h\bigr\rrVert _{\rho,t_0,K}>
\eta \bigr\}&=&0,
\end{eqnarray}
where $\eta$ is an arbitrary positive real number.\vadjust{\goodbreak}

This follows from the general approximation result developed in Section~\ref{s3}.
Indeed, consider equations \eqref{s3.7} and \eqref{s3.6}
with the choice of coefficients $A=D=0$, $B=\varsigma$. Then the
processes $X$ and $X_n$ coincide with $u$ and $\Phi^{w^n}$, respectively.
Hence, the convergence \eqref{sm.3} follows from Theorem~\ref{ts3.1}.
Next, we consider again equations \eqref{s3.7} and \eqref{s3.6} with
a new choice of coefficients: $A=D=\varsigma$, $B=-\varsigma$. In this
case, the processes $X$ and $X_n$ are equal to $\Phi^h$ and
$v_n:=u\circ T_n^h$, respectively. Thus, Theorem~\ref{ts3.1} yields
\eqref{sm.4}.\vspace*{1pt}
\end{proof}


\section{Auxiliary results}
\label{s4}

The most difficult part in the proof of Theorem~\ref{ts3.1} consists
of establishing \eqref{s3.15}. In particular, handling the
contribution of the pathwise integral (with respect to $w^n$) requires
a careful analysis of the discrepancy between this integral and the
stochastic integral with respect to $M$. This section gathers several
technical results that have been applied in the analysis of such
questions in the preceding Section~\ref{s3}.

The first statement in the next lemma provides a measure of the
discrepancy between the processes $X(t,x)$ and $X(t,t_n,x)$ defined in
\eqref{s3.7}, \eqref{s3.8.3}, respectively.\vspace*{1pt}
%
\begin{lemma}
\label{ls4.1}
Suppose that Hypothesis~\textup{\ref{HypB}} is satisfied. Then for any $p\in[1,\infty
)$ and every integer $n\ge1$,\vspace*{1pt}
%
\begin{equation}
\label{s4.4} \sup_{(t,x)\in[0,T]\times\R^3} \bigl \|X(t,x)-X(t,t_n,x)
\bigr \|_p\le C2^{-np\tfrac{\nu+1}{2}}
\end{equation}
and\vspace*{1pt}
%
\begin{equation}
\label{s4.5} \sup_{n\ge1}\sup_{(t,x)\in[0,T]\times\R^3}
\bigl \|X(t,t_n,x)\bigr \|_p\le C<\infty,
\end{equation}
where $C$ is a positive constant not depending on $n$.
\end{lemma}

\begin{proof} Fix $p\in[2,\infty[$. From equations \eqref{s3.7},
\eqref{s3.8.3}, we obtain\vspace*{1pt}
\[
\bigl \|X(t,x)-X(t,t_n,x)\bigr \|_p^p\le C
\bigl(V_1(t,x)+ V_2(t,x)+V_3(t,x) \bigr),
\]
where\vspace*{1pt}
\begin{eqnarray*}
V_1(t,x) &:=& \biggl\llVert \int_{t_n}^t
\int_{\R^3} G(t-s,x-y) (A+B) \bigl(X(s,y)\bigr) M(\mathrm{d}s,
\mathrm{d}y)\biggr\rrVert _p^p,
\\
V_2(t,x) &:=&\bigl\llVert G(t-\cdot,x-\ast)D\bigl(X(\cdot,\ast )
\bigr)1_{[t_n,t]}(\cdot),h\rangle_{\mathcal{H}_t}\bigr\rrVert
_p^p,
\\
V_3(t,x) &:=&\biggl\llVert \int_{t_n}^t
G(t-s,\cdot)\star b\bigl(X(s,\cdot)\bigr) (x) \,\mathrm{d}s\biggr\rrVert
_p^p.
\end{eqnarray*}
Applying first Burholder's and then H\"older's inequalities, we obtain
\begin{eqnarray*}
V_1(t,x) &\le& C \E\bigg(\int_{t_n}^t
\mathrm{d}s  \int_{\R^3}\int_{\R^3}G(t-s,x-dy) G(t-s,x-dz)f(y-z)\\
&&\qquad\times\big(A+B\big)\big(X(t,y)\big) \big(A+B)(X(t,z)\big)  \biggl)^{p/2}\\
&=&C \E\bigg(\int_0^{t-t_n}
\mathrm{d}s  \int_{\R^3}\int_{\R^3}G(s,x-dy) G(s,x-dz)f(y-z)\\
&&\qquad\times\big(A+B\big)\big(X(t-s,y)\big) \big(A+B)(X(t-s,z)\big)  \biggl)^{p/2}\\
&\le& C (t-t_n)^{\tfrac{p}{2}-1} \int_0^{t-t_n}
\mathrm{d}s \bigg(\int_{\R^3}\int_{\R^3}G(s,x-dy) G(s,x-dz)f(y-z)\biggl)^{p/2}\\
&&\times \sup_{(t,x)\in[0,T]\times\R^3}\E
\bigl(\bigl |(A+B) \bigl(X(t,x)\bigr)\bigr  |^{p} \bigr)\\
&\le& C (t-t_n)^{\tfrac{p}{2}-1} \int_0^{t-t_n}
\mathrm{d}s \bigg(\int_{\R^3}\int_{\R^3} G(s,x-dy) G(s,x-dz)f(y-z)\biggl)^{p/2}\\
&&\times \Bigl(1 + \sup_{(t,x)\in[0,T]\times\R^3}
\E \bigl( \bigl|X\bigl((t,x)\bigr|^{p}\bigr) \Bigr)
\end{eqnarray*}

Applying the inequality \eqref{s5.21}, Lemma 7.1 in \cite{h-hu-nu}  and then a) of the  {\bf hypothesis 2}, we get 
\begin{eqnarray*}
V_1(t,x) &\le& C (t-t_n)^{\tfrac{p}{2}-1} \int_0^{t-t_n}
\mathrm{d}s \bigg(\int_{|z|\le 2s}\frac{f(z)}{|z|}\biggl)^{p/2}\\
&\le& C (t-t_n)^{\tfrac{p}{2}-1} \int_0^{t-t_n}
s^{\nu p/2} \mathrm{d}s \le C (t-t_n)^{p\tfrac{\nu+1}{2}}\le C 2^{-np\tfrac{\nu+1}{2}}.
\end{eqnarray*}

For the study of $V_2$, we apply first Cauchy--Schwarz inequality and
then H\"older's inequality. We obtain
\[
V_2(t,x) \le\bigl \|h1_{[t_n,t]}(\cdot)\bigr \|_{\mathcal{H}_t}^{{\trup{p}{2}}}
\E \biggl(\int_0^t \mathrm{d}s \bigl \|G(t-s,x-
\ast)D\bigl(X(s,\ast)\bigr) 1_{[t_n,t]}(s)\bigr \|_{\mathcal{H}}^2
\biggr)^{{\trup{p}{2}}}.
\]
Hence, similarly as for $V_1$ we have
\[
V_2(t,x)\le C  2^{-np\tfrac{\nu+1}{2}}.
\]

By applying H\"older's inequality, we get
\begin{eqnarray*}
V_3(t,x) &\le& \biggl(\int_{t_n}^t
\mathrm{d}s \int_{\R^3} G(t-s,x-\mathrm{d}y)
\biggr)^{p-1} \int_{t_n}^t \mathrm{d}s
\int_{\R^3} G(t-s,x-\mathrm{d}y) \E\bigl(\bigl |b\bigl(X(s,y)
\bigr)\bigr |^p\bigr)
\\
&\le& C \biggl(\int_{t_n}^t \mathrm{d}s \int
_{\R^3} G(t-s,x-\mathrm{d}y) \biggr)^p \Bigl(1+
\sup_{(t,x)\in[0,T]\times\R^3} \E\bigl(\bigl |X(s,y)\bigr |^p\bigr) \Bigr)
\\
&\le& C 2^{-2np}.
\end{eqnarray*}
The condition $\nu\in\,]0,1[$ implies $2^{-2np}< 2^{-np\tfrac{\nu+1}{2}}$. Thus from the estimates on $V_i(t,x)$, $i=1,2,3$
(which hold uniformly on $(t,x)\in[0,T]\times\IR^3$) we obtain
\eqref{s4.4}.

Finally, \eqref{s4.5} is a consequence of the triangular inequality,
\eqref{s4.4} and \eqref{s5.21}.
\end{proof}


The next result states an analogue of Lemma~\ref{ls4.1} for the
stochastic processes $X_n$, $X_n^-$ defined in \eqref{s3.6}, \eqref
{s3.8.2}, respectively, this time including a localization by $L_n$.

\begin{lemma}
\label{ls4.2}
We assume Hypothesis~\textup{\ref{HypB}}.
Then for any $p\in[2,\infty)$ and $t\in[0,T]$,
%
\begin{eqnarray}
\label{s4.11} &&\sup_{(s,y)\in[0,t]\times\R^3}\E \bigl(\bigl |X_n(s,y)-X_n^-(s,y)\bigr |^p1_{L_n(s)}
\bigr)
\nonumber
\\[-8pt]
\\[-8pt]
&&\quad \le Cn^{\trup{3p}{2}} 2^{-np\tfrac{\nu+1}{2}} \Bigl[1+ \sup
_{(s,y)\in[0,t]\times\R^3} \E\bigl(\bigl |X_n(s,y)\bigr |^p1_{L_n(s)}
\bigr) \Bigr].
\nonumber
\end{eqnarray}
\end{lemma}

\begin{proof}
Fix $p\in[2,\infty[$ and consider the decomposition
%
\begin{equation}
\label{s4.12} \E \bigl(\bigl |X_n(t,x) - X_n^-(t,x)\bigr |^p
1_{L_n(t)} \bigr)\le C\sum_{i=1}^4
T_{n,i}^k (t,x),
\end{equation}
where
\begin{eqnarray*}
T_{n,1}(t,x) &=&\E \biggl(\biggl\llvert \int_{t_n}^t
\int_{\R^3} G(t-s,x-y) A\bigl( X_n(s,y)\bigr)M(
\mathrm{d}s,\mathrm{d}y) \biggr\rrvert ^p 1_{L_n(t)} \biggr),
\\
T_{n,2}(t,x) &=&\E \bigl(\bigl\llvert \bigl\langle G(t-\cdot,x-\ast) B
\bigl( X_n(\cdot,\ast)\bigr)1_{[t_n,t]}(\cdot),w^n
\bigr\rangle_{\mathcal{H}_t} \bigr\rrvert ^p 1_{L_n(t)} \bigr),
\\
T_{n,3}(t,x) &=&\E \bigl(\bigl\llvert \bigl\langle G(t-\cdot,x-\ast) D
\bigl( X_n(\cdot,\ast)\bigr)1_{[t_n,t]}(\cdot),h\bigr
\rangle_{\mathcal{H}_t} \bigr\rrvert ^p 1_{L_n(t)} \bigr),
\\
T_{n,4}(t,x) &=&\E \biggl(\biggl\llvert \int_{t_n}^t
G(t-s,\cdot)\star b\bigl(X_n(s,\cdot)\bigr) (x)\, \mathrm{d}s \biggr
\rrvert ^p 1_{L_n(t)} \biggr).
\end{eqnarray*}

By the same arguments used for the analysis of $V_1(t,x)$ in the
preceding lemma, we obtain
%
\begin{equation}
\label{s4.13} T_{n,1}(t,x)\le C 2^{-np\tfrac{\nu+1}{2}}\times \Bigl[1+\sup
_{(s,y)\in[0,t]\times\R^3}\E \bigl(\bigl\llvert X_n(s,y) \bigr\rrvert
^p 1_{L_n(s)} \bigr) \Bigr].
\end{equation}


For $T_{n,2}(t,x)$, we first use Cauchy--Schwarz' inequality to obtain
\begin{eqnarray*}
&&T_{n,2}(t,x)
\\
&&\quad\le\E \bigl(\bigl\llvert \bigl \|w^n
1_{[t_n,t]} 1_{L_n(t)}\bigr \| _{\mathcal{H}_t} \bigl \|G(t-\cdot,x-\ast) B
\bigl( X_n(\cdot,\ast)\bigr)1_{[t_n,t]}(\cdot )1_{L_n(t)}
\bigr \|_{\mathcal{H}_t}\bigr\rrvert ^p \bigr).
\end{eqnarray*}
Appealing to \eqref{s3.14}, this yields
\begin{eqnarray*}
T_{n,2}(t,x)&\le& C n^{\trup{3p}{2}} \E \biggl(\biggl\llvert \int
_{t_n}^t \mathrm{d}s \bigl \|G(t-s,x-\ast) B\bigl(
X_n(s,\ast)\bigr) (s)1_{L_n(s)}\bigr \|_{\mathcal{H}}^2
\biggr\rrvert ^{{\trup{p}{2}}} \biggr)\\
&=& C n^{\trup{3p}{2}} \E\bigg(\int_{t_n}^t
\mathrm{d}s  \int_{\R^3}\int_{\R^3}G(t-s,x-dy) G(t-s,x-dz)f(y-z)\\
&&\qquad\qquad\times B\big(X(t,y)\big) B\big(X(t,z)\big)  \biggl)^{p/2}. 
\end{eqnarray*}

We can now proceed as for the term $V_2((t,x)$ in the proof of Lemma~\ref{ls4.1}. We obtain
%
\begin{equation}
\label{s4.14} T_{n,2}(t,x)\le C n^{\trup{3p}{2}} 2^{-np\tfrac{\nu+1}{2}} 
\Bigl[1+\sup_{(s,y)\in[0,t]\times\R^3}\E \bigl(\bigl\llvert X_n(s,y)
\bigr\rrvert ^p 1_{L_n(s)} \bigr) \Bigr].
\end{equation}

The difference between the terms $T_{n,3}(t,x)$ and $T_{n,2}(t,x)$ is
that $w^n$ in the latter is replaced by $h$ in the former. Hence,
following similar arguments as for the study of $T_{n,2}(t,x)$, and
using that $\Vert h 1_{[t_n,t]} 1_{L_n(t)}\Vert_{\mathcal
{H}_T}<\infty$, we prove
%
\begin{equation}
\label{s4.15} T_{n,3}(t,x)\le C 2^{-np\tfrac{\nu+1}{2}} \times \Bigl[1+\sup
_{(s,y)\in[0,t]\times\R^3}\E \bigl(\bigl\llvert X_n(s,y) \bigr\rrvert
^p 1_{L_n(s)} \bigr) \Bigr].
\end{equation}

Finally, we notice the similitude between $T_{n,4}(t,x)$ and $V_3(t,x)$
in Lemma~\ref{ls4.1}. Proceeding as for the study of this term, we obtain
%
\begin{eqnarray}
\label{s4.16} T_{n,4}(t,x)&\le& C \biggl(\int_{t_n}^t
\mathrm{d}s \int_{\IR^3} G(t-s,x-\mathrm{d}y)
\biggr)^p \Bigl[1+\sup_{(s,y)\in[0,t]\times\R^3} \E \bigl(\bigl\llvert
X_n(s,y)\bigr\rrvert ^p 1_{L_n(s)} \bigr) \Bigr]
\nonumber
\\[-8pt]
\\[-8pt]
&\le& C 2^{-np\tfrac{\nu+1}{2}} \Bigl[1+\sup_{(s,y)\in[0,t]\times
\R^3}\E \bigl(\bigl
\llvert X_n(s,y)\bigr\rrvert ^p 1_{L_n(s)} \bigr)
\Bigr].
\nonumber
\end{eqnarray}
From \eqref{s4.12}--\eqref{s4.16} we obtain \eqref{s4.11}.
\end{proof}

%
\begin{lemma}
\label{ls4.3}
We assume Hypothesis~\textup{\ref{HypB}}. Then, for any $p\in[1,\infty)$, there
exists a finite constant $C$ such that
%
\begin{equation}
\label{s4.18} \sup_{n\ge1}\sup_{(t,x)\in[0,T]\times\R^3}\E
\bigl[\bigl(\bigl |X_n(t,x)\bigr |^p+\bigl |X_n^-(t,x)\bigr |^p
\bigr)1_{L_n(t)} \bigr]\le C.
\end{equation}
Moreover,
%
\begin{equation}
\label{s4.19} \sup_{(t,x)\in[0,T]\times\R^3} \bigl \|\bigl(X_n(t,x)-X_n^-(t,x)
\bigr)1_{L_n(t)}\bigr \| _p \le C n^{\trup{3}{2}} 2^{-n\tfrac{\nu+1}{2}}.
\end{equation}
\end{lemma}
\begin{proof}
For $0\le r\le t$, define
\begin{eqnarray*}
X_n(t,r;x) &=&\frac{d}{dt}\big(G(t)*v_0 \big)(x) + \big(G(t)*\bar v_0 \big)(x)
\\
&&{} \int_0^r
\int_{\R^3} G(t-s,x-y) A\bigl(X_n(s,y)\bigr) M(
\mathrm{d}s,\mathrm{d}y)
\\
&&{} +\bigl\langle G(t-\cdot,x-\ast)B\bigl(X_n(\cdot,\ast)
\bigr)1_{[0,r]}(\cdot ),w^n\bigr\rangle_{\mathcal{H}_t}
\\
&&{} +\bigl\langle G(t-\cdot,x-\ast)D\bigl(X_n(\cdot,\ast)
\bigr)1_{[0,r]}(\cdot ),h\bigr\rangle_{\mathcal{H}_t} +\int
_0^r G(t-s,\cdot)\star b\bigl(X_n(s,
\cdot)\bigr) (x) \,\mathrm{d}s.
\end{eqnarray*}

Fix $p\in[2,\infty[$ and consider the decomposition
\[
\E\bigl(\bigl |X_n(t,r;x)\bigr |^p1_{L_n(t)}\bigr)\le C \sum
_{i=1}^5 T_{n,i}(t,r;x),
\]
where
\begin{eqnarray*}
T_{n,0}(t,r;x) &=& \Big| \frac{d}{dt}\big(G(t)*v_0 \big)(x) + \big(G(t)*\bar v_0 \big)(x)   \Big|^p \\
T_{n,1}(t,r;x) &=& \E \biggl(\biggl\llvert \int_0^r
\int_{\R^3} G(t-s,x-y)A\bigl(X_n(s,y)\bigr) M(
\mathrm{d}s,\mathrm{d}y)\biggr\rrvert ^p 1_{L_n(t)} \biggr),
\\
T_{n,2}(t,r;x) &=& \E \bigl(\bigl\llvert \bigl\langle G(t-\cdot,x-
\ast) B\bigl(X_n^-(\cdot,\ast)\bigr) 1_{[0,r]}(
\cdot),w^n\bigr\rangle_{\mathcal{H}_t} \bigr\rrvert ^p
1_{L_n(t)} \bigr),
\\
T_{n,3}(t,r;x) &=& \E \bigl(\bigl\llvert \bigl\langle G(t-\cdot,x-
\ast) \bigl[B\bigl(X_n(\cdot,\ast)\bigr)-B\bigl(X_n^-(
\cdot,\ast)\bigr)\bigr] 1_{[0,r]}(\cdot ),w^n\bigr
\rangle_{\mathcal{H}_t}\bigr\rrvert ^p1_{L_n(t)} \bigr),
\\
T_{n,4}(t,r;x) &=& \E \bigl(\bigl\llvert \bigl\langle G(t-\cdot,x-
\ast) D\bigl(X_n(\cdot,\ast)\bigr) 1_{[0,r]}(\cdot),h\bigr
\rangle_{\mathcal{H}_t}\bigr\rrvert ^p1_{L_n(t)} \bigr),
\\
T_{n,5}^k (t,r;x) &=& \E \biggl(\biggl\llvert \int
_0^rG(t-s,\cdot)\star b\bigl(X_n(s,
\cdot)\bigr) (x) \,\mathrm{d}s \biggr\rrvert ^p 1_{L_n(t)}
\biggr).
\end{eqnarray*}

By Hypothesis on $v,\bar v$ and for the Theorem 4.6 in \cite{dalang-quer}, we have that 

\begin{eqnarray}
\label{s4.22} T_{n,0}(t,r;x) &\le& C 
\end{eqnarray}


Similarly as for the term $V_1(t,x)$ in Lemma~\ref{ls4.1}, we have
%
\begin{eqnarray}
\label{s4.22.1} T_{n,1}(t,r;x) &\le& C \biggl(\int
_0^r \mathrm{d}s \int_{\IR^3}
\mu(\mathrm{d}\xi)\bigl \vert\mathcal{F} G(t-s) (\xi)\bigr \vert^2
\biggr)^{\trup{p}{2}-1}
\nonumber
\\
&&{} \times\int_0^r \mathrm{d}s \Bigl[1+
\sup_{(\hat{s},y)\in[0,s]\times\R
^3}\E \bigl(\bigl\llvert X_n(\hat{s},y)
\bigr\rrvert ^{p}1_{L_n(\hat
{s})} \bigr) \Bigr]
\nonumber
\\[-8pt]
\\[-8pt]
&&{}\times\biggl(\int
_{\IR^3}\mu(\mathrm{d}\xi)\bigl \vert\mathcal{F} G(t-s) (\xi)\bigr \vert
^2 \biggr)\nonumber
\\
& \le& C \int_0^r \mathrm{d}s \Bigl[1+ \sup
_{(\hat{s},y)\in[0,s]\times\R
^3}\E \bigl(\bigl\llvert X_n(\hat{s},y) \bigr
\rrvert ^{p}1_{L_n(\hat{s})} \bigr) \Bigr].
\nonumber
\end{eqnarray}

Let $\tau_n$ and $\pi_n$ be as in the proof of Lemma~\ref{lss3.1.2}
(see \eqref{s3.34} and the successive lines). Since $X_n^-(s,y)$ is
$\mathcal{F}_{s_n}$-measurable, the definition of $w^n$ implies
\begin{eqnarray*}
&& T_{n,2}(t,r;x)
\\
&&\quad =\E \biggl(\biggl\llvert \int_0^t\int
_{\R^3} (\pi_n \circ\tau_n)
\bigl[G(t-\cdot,x-\ast)B\bigl(X_n^-(\cdot,\ast)\bigr)
\bigr](s,y)1_{L_n(t)} 1_{[0,r]}(\cdot)M(\mathrm{d}s,\mathrm{d}y)
\biggr\rrvert ^p \biggr).
\end{eqnarray*}
Then, applying Burkholder's inequality, using the boundedness of the
operator $\pi_n \circ\tau_n$, and similar arguments as for the term
$T_{n,1}(t,r;x)$ we obtain
%
\begin{equation}
\label{s4.24} T_{n,2}(t,r;x) \le C \int_0^r
\mathrm{d}s \Bigl[1+ \sup_{(\hat{s},y)\in[0,s]\times\R
^3}\E \bigl(\bigl\llvert
X_n^-(\hat{s},y) \bigr\rrvert ^{p}1_{L_n(\hat{s})}
\bigr) \Bigr].
\end{equation}

To study $T_{n,3}(t,r;x)$, we apply Cauchy--Schwarz and then H\"older's
inequality. This yields
\begin{eqnarray*}
&& T_{n,3}(t,r;x)
\\
&&\quad \le\E \bigl(\bigl\llvert \bigl\llVert w^n 1_{[0,r]}
1_{L_n(t)}\bigr\rrVert _{\mathcal{H}_t}^2\bigl\llVert G(t-
\cdot,x-\ast) \bigl[B( X_n)-B\bigl( X_n^-\bigr) \bigr](
\cdot,\ast) 1_{[0,r]}(\cdot) 1_{L_n(t)}\bigr\rrVert
_{\mathcal{H}_t}^2 \bigr\rrvert ^{{\trup{p}{2}}} \bigr)
\\
&&\quad \le C n^{\trup{3p}{2}}2^{n{\trup{p}{2}}} \E \biggl(\int
_0^t \mathrm{d}s \bigl\llVert G(t-s,x-\ast)
\bigl[B( X_n)-B\bigl(X_n^-\bigr) \bigr](s,\ast)
1_{[0,r]}(s)1_{L_n(s)} \bigr\rrVert _{\mathcal{H}}^2
\biggr)^{{\trup{p}{2}}}
\\
&&\quad \le C n^{\trup{3p}{2}} 2^{n{\trup{p}{2}}} \biggl(\int_0^r
\mathrm{d}s\int_{\R^3} \mu(\mathrm{d}\xi)\bigl \vert\mathcal
{F}G(t-s)\bigr \vert^2(\xi) \biggr)^{\trup{p}{2}-1}
\\
&&\qquad{} \times\int_0^r \mathrm{d}s \sup
_{(\hat{s},y)\in[0,s]\times\R^3}\E \bigl(\bigl\llvert X_n(
\hat{s},y)-X_n^-(\hat{s},y) \bigr\rrvert ^p
1_{L_n(\hat{s})} \bigr)
\\
&&\qquad{}\times \biggl(\int_{\R^3} \mu(\mathrm{d}\xi)
\bigl \vert\mathcal{F}G(t-s)\bigr \vert^2(\xi ) \biggr),
\end{eqnarray*}
where we have used \eqref{s3.101} and the Lipschitz continuity of the
function $B$. By applying \eqref{s4.11}, we obtain
\begin{eqnarray*}
T_{n,3}(t,r;x) &\le& C n^{3p}  2^{-np\tfrac{\nu+1}{2}-np/2} \int
_0^r \mathrm{d}s \Bigl[ 1 + \sup
_{(\hat{s},y)\in[0,s]\times\R
^3}\E \bigl(\bigl\llvert X_n(\hat{s},y)\bigr
\rrvert ^p 1_{L_n(\hat{s})} \bigr) \Bigr]
\\
&\le& C n^{3p}  2^{-np\nu/2} \int_0^r \mathrm{d}s \Bigl[ 1 + \sup
_{(\hat{s},y)\in[0,s]\times\R^3}\E \bigl(\bigl\llvert X_n(\hat{s},y)\bigr
\rrvert ^p 1_{L_n(\hat{s})} \bigr) \Bigr],
\end{eqnarray*}
where in the last inequality we have used that
$\sup_n  \{n^{3p} 2^{-np\nu/2}\}<\infty$.

We now consider $T_{n,4}(t,r;x)$. With similar arguments as those used
in the analysis of $T_{n,3}(t,x)$ in Lemma~\ref{ls4.2}, we prove
%
\begin{equation}
\label{s4.26} T_{n,4}(t,r;x)\le C \int_0^r
\mathrm{d}s \Bigl[1+ \sup_{(\hat{s},y)\in
[0,s]\times\R^3}\E \bigl(\bigl\llvert
X_n(\hat{s},y) \bigr\rrvert ^{p}1_{L_n(\hat{s})} \bigr)
\Bigr].
\end{equation}
Finally, we notice that $T_{n,5}(t,r;x)$ is very similar to
$T_{n,4}(t,x)$ in Lemma~\ref{ls4.2}. With similar arguments as those
used in the analysis of this term, we have
%
\begin{equation}
\label{s4.27} T_{n,5}(t,r;x) \le C \int_0^r
\mathrm{d}s \Bigl[ 1 + \sup_{(\hat{s},y)\in
[0,s]\times\R^3} \E \bigl(\bigl\llvert
X_n(\hat{s},y) \bigr\rrvert ^p 1_{L_n(\hat{s})} \bigr)
\Bigr].
\end{equation}

Bringing together
\eqref{s4.22}, \eqref{s4.24}--\eqref{s4.27} yields
%
\begin{eqnarray}
\label{s4.28}
&&\E\bigl(\bigl |X_n(t,r;x)\bigr |^p1_{L_n(t)}
\bigr)
\nonumber
\\[-8pt]
\\[-8pt]
&&\quad\le C \biggl\{1+ \int_0^r \sup
_{(\hat
{s},y)\in[0,s]\times\R^3} \E \bigl( \bigl\{\bigl\llvert X_n(\hat{s},y)
\bigr\rrvert ^p+ \bigl\llvert X_n^-(\hat {s},y) \bigr
\rrvert ^p \bigr\}1_{L_n(\hat{s})} \bigr)\,\mathrm{d}s \biggr\}.
\nonumber
\end{eqnarray}
Notice that $X_n(t,t;x)=X_n(t,x)$. Hence, for $r:=t$, \eqref{s4.28}
tells us
%
\begin{eqnarray}
\label{s4.29}
&& \E\bigl(\bigl |X_n(t,x)\bigr |^p1_{L_n(t)}
\bigr)
\nonumber
\\[-8pt]
\\[-8pt]
&&\quad\le C \biggl\{1+ \int_0^t \sup
_{(\hat
{s},y)\in[0,s]\times\R^3} \E \bigl( \bigl\{\bigl\llvert X_n(\hat{s},y)
\bigr\rrvert ^p +\bigl\llvert X_n^-(\hat{s},y) \bigr
\rrvert ^p \bigr\}1_{L_n(\hat{s})} \bigr) \,\mathrm{d}s \biggr\}.\nonumber
\end{eqnarray}
Next, take $r:=t_n$ and remember that $X_n(t,t_n;x)=X_n^-(t,x)$. From
\eqref{s4.28}, and since $t_n\le t$, we obtain
%
\begin{eqnarray}
\label{s4.30}
&&\E\bigl(\bigl |X_n^-(t,x)\bigr |^p1_{L_n(t)}
\bigr)
\nonumber
\\[-8pt]
\\[-8pt]
&&\quad\le C \biggl\{1+ \int_0^t \sup
_{(\hat{s},y)\in[0,s]\times\R^3} \E \bigl( \bigl\{\bigl\llvert X_n(\hat{s},y)
\bigr\rrvert ^p+\bigl\llvert X_n^-(\hat {s},y) \bigr
\rrvert ^p \bigr\}1_{L_n(\hat{s})} \bigr) \,\mathrm{d}s \biggr\}.\nonumber
\end{eqnarray}
For $t\in[0,T]$, set
\[
\varphi_n(t)= \sup_{(s,y)\in[0,t]\times\R^3} \E \bigl[
\bigl(\bigl |X_n(s,y)\bigr |^p+\bigl |X_n^-(s,y)\bigr |^p
\bigr)1_{L_n(s)} \bigr].
\]
The inequalities \eqref{s4.29}, \eqref{s4.30} imply
$\varphi_n(t)\le C \{1+ \int_0^t \varphi_n(s)\,\mathrm{d}s \}$.
By Gronwall's lemma, this implies \eqref{s4.18}. Finally, the
inequality \eqref{s4.19} is a consequence of \eqref{s4.11} and \eqref{s4.18}
\end{proof}


\section{Appendix}
\label{s5}

We start this section with a theorem on existence and uniqueness of
solution to a class of equations which in particular applies to \eqref
{s3.7}, and therefore also to \eqref{s1.8}, and to
\eqref{s3.6}. For related results, we refer the reader to \cite
{dalang}, Theorem~13, \cite{dalang-quer}, Theorem~4.3.
In comparison with these references, here we state the theorem in
spatial dimension $d=3$, and we assume that $G$ is the fundamental
solution of the wave equation in dimension three.
\begin{theorem}
\label{ts5.1}
Let $G$ denote the fundamental solution to the wave equation in
dimension three and $M$ a Gaussian process as given in the
\hyperref[s1]{Introduction}. Consider the stochastic evolution equation defined by
%
\setcounter{equation}{0}
\begin{eqnarray}
\label{s5.20} Z(t,x)&=&\frac{d}{dt}\big(G(t)*v_0 \big)(x) + \big(G(t)*\bar v_0 \big)(x)
\\
&&{}+\int_0^t\int
_{\IR^3}G(t-s,x-y) \varsigma \bigl(Z(s,y)\bigr)M(\mathrm{d}s,
\mathrm{d}y)
\nonumber
\\
&&{}+\bigl\langle G(t-\cdot,x-\ast)g\bigl(Z(\cdot,\ast)\bigr),H\bigr
\rangle_{\mathcal
{H}_T}
\\
&&{}+\int_0^t \bigl[G(t-s,
\cdot)\star b\bigl(Z(s,\cdot)\bigr)\bigr](x),
\nonumber
\end{eqnarray}
where the functions $\varsigma, g, b: \IR\rightarrow\IR$ are Lipschitz
continuous.
\begin{enumerate}
\item[(i)] Assume that $H=\{H_t, t\in[0,T]\}$ is an $\hac$-valued
predictable stochastic process such that
$C_0:=\sup_\omega\Vert H(\omega)\Vert_{\mathcal{H}_T} <\infty$.

Then, there exists a unique real-valued adapted stochastic process $Z=\{
Z(t,x),\allowbreak   (t,x)\in[0,T]\times\IR^3\}$ satisfying \textup{\eqref{s5.20}}, a.s.,
for all $(t,x)\in[0,T]\times\R^3$. Moreover, the process $Z$ is
continuous in $L^2$ and satisfies
\[
\sup_{(t,x)\in[0,T]\times\IR^3}\E \bigl(\bigl \vert Z(t,x)\bigr \vert^p \bigr) \le
C<\infty,
\]
for any $p\in[1,\infty[$, where the constant $C$ depends among others
on $C_0$.
\item[(ii)] Assume that there exist an increasing sequence of events
$\{\Omega_n, n\ge1\}$ such that $\lim_{n\to\infty}\mathbb
{P}(\Omega_n)=1$, and that $H_n=\{H_n(t), t\in[0,T]\}$ is a sequence
of $\hac$-valued predictable stochastic processes such that
$C_n:=\sup_\omega\Vert H(\omega)1_{\Omega_n}(\omega)\Vert
_{\mathcal{H}_T} <\infty$. Then, the conclusion on existence and
uniqueness of solution to \textup{\eqref{s5.20}} stated in part (\textup{i}) also holds.
\end{enumerate}
The process $Z$ is termed a \textit{random field solution} to \textup{\eqref{s5.20}}.
\end{theorem}

\begin{proof}(Sketch of the proof)
We start with part (i). Consider the
Picard iteration scheme
\[
Z^0(t,x) =\frac{d}{dt}\big(G(t)*v_0 \big)(x) + \big(G(t)*\bar v_0 \big)(x),\vadjust{\goodbreak}
\]
\begin{eqnarray*}
Z^{(k+1)}(t,x) &=&\frac{d}{dt}\big(G(t)*v_0 \big)(x) + \big(G(t)*\bar v_0 \big)(x)
\\
&&{}+\int_0^t\int
_{\IR^3}G(t-s,x-y) \varsigma \bigl(Z^{(k)}(s,y)\bigr)M(
\mathrm{d}s,\mathrm{d}y)
\\
&&{}+\bigl\langle G(t-\cdot,x-\ast)g\bigl(Z^{(k)} (\cdot,\ast)\bigr),H
\bigr\rangle _{\mathcal{H}_T}
 +\int_0^t
\bigl[G(t-s,\cdot)\star b\bigl(Z^{(k)} (s,\cdot)\bigr)\bigr](x),
\end{eqnarray*}
$k\ge0$.

Fix $p\in[2,\infty[$. First, we prove by induction on $k\ge0$ that
%
\begin{equation}
\sup_{(t,x)\in[0,T]\times\IR^3}\E \bigl(\bigl \vert Z^{k}(t,x)\bigr \vert
^p \bigr)\le C<\infty,
\end{equation}
with a constant $C$ independent of $k$.
Second, we prove that
\begin{eqnarray*}
&&\sup_{x\in\IR^3}\E \bigl(\bigl \vert Z^{(k+1)}(t,x)-Z^{(k)}(t,x)
\bigr \vert ^p \bigr)
\\
&&\quad \le C(1+C_0) \biggl[\int_0^t
\mathrm{d}s \sup_{y\in\IR^3}\E \bigl(\bigl \vert Z^{(k)}(s,y)-Z^{(k-1)}(s,y)
\bigr \vert^p \bigr) \biggr].
\end{eqnarray*}
With this, we conclude that the sequence of processes $\{Z^{(k)}(t,x),
(t,x)\in[0,T]\times\IR^3\}$, $k\ge0$ converges in $L^p(\Omega)$ as
$k\to\infty$, uniformly
in $(t,x)\in[0,T]\times\IR^3$. The limit is a random field that
satisfies the properties of the statement. We refer the reader to
\cite{dalang,dalang-quer}, for more
details on the proof.

The proof of part (ii) is done by localizing the preceding Picard
scheme using the sequence $\{\Omega_n, n\ge1\}$.
\end{proof}

In comparison with the equation considered in \cite{dalang-quer}, Theorem~4.3, \eqref{s5.20} has the
extra term
$\langle G(t-\cdot,x-\ast)g(Z(\cdot,\ast)),H\rangle_{\mathcal{H}_T}$.

Part (i) of Theorem~\ref{ts5.1} can be applied to \eqref{s1.8},
\eqref{s3.7}. Therefore, we have
%
\begin{equation}
\label{s5.21} \sup_{(t,x)\in[0,T]\times\IR^3}\E \bigl(\bigl \vert X(t,x)
\bigr \vert^p \bigr) <\infty.
\end{equation}
Let $\Omega_n=L_n(t)$ as given in \eqref{localization}. The sequence
$H_n:=w^n$ defined in \eqref{s3.3} satisfies the assumptions of part
(ii) of Theorem~\ref{ts5.1} (see \eqref{s3.101}). Therefore the
conclusion applies to the stochastic process solution of \eqref{s3.6}.

%
\begin{remark}
\label{r5.1}
Set $Z^{(z)}(s,x)=Z(s,x+z)$, $z\in\R^3$. In opposition to \cite
{dalang}, we cannot argue that the finite dimensional distributions of the
process $\{Z^{(z)}(s,x), (s,x)\in[0,T]\times\R^3\}$ do not depend on
$z$. This is a consequence from the fact that the initial condition
of the SPDE is not zero.
\end{remark}

In the proof of Theorem~\ref{ts3.1}, we have used the lemma below. For
its proof, we refer the reader to Lemma A.2 in \cite{milletss2},
with a trivial change on the spacial dimension
($d=3$ in \cite{milletss2}, while $d=4$ in Lemma~\ref{ls5.2}).

\begin{lemma}
\label{ls5.2}
Fix $[t_0,T]$ with $t_0\ge0$ and a compact set $K\subset\IR^3$.
Let $\{Y_n(t,x),(t,x)\in[t_0,T]\times K, n\ge1\}$ be a sequence of
processes and
$\{B_n(t), t\in[t_0,T]\}\subset\mathcal{F}$ be a sequence of adapted
events which, for every n, decreases in $t$. Assume that for every
$p\in\,]1,\infty[$ the following conditions hold:
\begin{enumerate}
\item[(P1)] There exists $\delta>0$ and $C>0$ such that, for any
$t_0\le t\le\bar{t}\le T$, $x,\bar{x}\in K$,
\[
\sup_n \E \bigl(\bigl |Y_n(t,x)-Y_n(
\bar{t},\bar{x})\bigr |^p 1_{B_n(\bar{t})} \bigr)\le C \bigl(|t-\bar{t}|+|x-
\bar{x}| \bigr)^{4+\delta}.
\]
\item[(P2)] For every $(t,x)\in[t_0,T]\times K$,
\[
\lim_{n\to\infty} \E \bigl(\bigl |Y_n(t,x)\bigr |^p
1_{B_n(t)} \bigr)=0.
\]
\end{enumerate}

Then, for any $\eta\in\,]0,\delta/p[$ and any $r\in[1,p[$,
\[
\lim_{n\to\infty} \E\bigl(\|Y_n\|_{\eta,t_0,K}^r
1_{B_n(T)} \bigr) =0.
\]
\end{lemma}


\end{document}